\documentclass[review,onefignum,onetabnum]{siamart190516}

\usepackage{amssymb}
\usepackage{algpseudocode}
\usepackage{bm}
\usepackage{booktabs}
\usepackage{capt-of}
\usepackage{color, colortbl}
\usepackage{multirow}
\usepackage{siunitx}
\usepackage{subcaption}
\usepackage[sort,compress]{cite}

\usepackage{lipsum}
\usepackage{amsfonts}
\usepackage{graphicx}
\usepackage{epstopdf}
\ifpdf
  \DeclareGraphicsExtensions{.eps,.pdf,.png,.jpg}
\else
  \DeclareGraphicsExtensions{.eps}
\fi

\usepackage{enumitem}
\setlist[enumerate]{leftmargin=.5in}
\setlist[itemize]{leftmargin=.5in}

\newsiamremark{remark}{Remark}
\newsiamremark{hypothesis}{Hypothesis}
\newsiamremark{example}{Example}
\crefname{hypothesis}{Hypothesis}{Hypotheses}
\newsiamthm{claim}{Claim}

\Crefname{subsection}{Section}{Sections}

\headers{Randomized Cholesky Factorization}{C. Chen, T. Liang and G. Biros}

\title{RCHOL: Randomized Cholesky Factorization for Solving SDD Linear Systems}

\author{Chao Chen\thanks{University of Texas at Austin, United States (\email{chenchao.nk@gmail.com},\email{liangty1998@hotmail.com}, \email{biros@oden.utexas.edu}).}
\and Tianyu Liang\footnotemark[2]
\and George Biros\footnotemark[3]}

\usepackage{amsopn}
\DeclareMathOperator{\diag}{diag}

\ifpdf
\hypersetup{
  pdftitle={Randomized Cholesky Factorization},
}
\fi

\newcommand{\rchol}{\texttt{rchol}}
\newcommand{\ichol}{\texttt{ichol}}

\newcommand{\bb}{\bm{b}}
\newcommand{\be}{\bm{e}}
\newcommand{\bone}{\bm{1}}

\newcommand{\G}{ {\mathcal{G}} }
\newcommand{\bigO}{ {\mathcal{O}} }

\newcommand{\cyan}[1]{{\textcolor{cyan}{#1}}}

\definecolor{Gray}{gray}{0.9}
\newcolumntype{g}{>{\columncolor{Gray}}c}

\newcolumntype{H}{>{\setbox0=\hbox\bgroup}c<{\egroup}@{}}

\graphicspath{{figs/}}

\nolinenumbers

\begin{document}
\maketitle

\begin{abstract}
    We introduce a randomized algorithm, namely \rchol{}, to construct an approximate Cholesky factorization for a given Laplacian matrix (a.k.a., graph Laplacian). From a graph perspective, the exact Cholesky factorization introduces a clique in the underlying graph after eliminating a row/column. By randomization, \rchol{}  only retains a sparse subset of the edges in the clique {using a random sampling developed by Spielman and Kyng~\cite{sampling}}.
    We prove \rchol{} is breakdown free and apply it to solving  large sparse linear systems with  symmetric diagonally-dominant matrices.
     In addition, we parallelize \rchol{} based on  the nested dissection ordering for shared-memory machines. We report numerical experiments that demonstrate the robustness and the scalability of \rchol{}. For example, 
     our parallel code scaled up to 64 threads on a single node for solving the  3D Poisson equation, discretized with the 7-point stencil on a $1024\times 1024 \times 1024$ grid, a problem that has \emph{one billion} unknowns.
\end{abstract}



\section{Introduction}

We consider the solution of a large sparse linear system 
\begin{equation} \label{e:axb}
    A \, x = b,
\end{equation}
where $A = (a_{ij})\in\mathbb{R}^{N\times N}$ is a symmetric diagonally-dominant (SDD) matrix, i.e., 
\begin{equation} \label{e:sdd}
A=A^\top, \quad \text{and} \quad a_{ii} \ge \sum_{j\not=i} |a_{ij}| \quad \text{for } i=1,2,\ldots,N.
\end{equation}
Note we require the diagonal of an SDD matrix to be non-negative\footnote{A relaxed definition requires  $|a_{ii}| \ge \sum_{j\not=i} |a_{ij}|$ allowing negative diagonal entries. This relaxed definition is \emph{not} what we use in this paper.}.  The linear system \cref{e:axb} appears  in many scientific and engineering domains, e.g., the discretization of a partial differential equation (PDE) using finite difference or finite elements, spectral graph partitioning, and learning problems on graphs.

The essential ingredient of our method is the randomized Cholesky factorization (\rchol). When $A$ has only negative nonzero off-diagonal entries 
,  \rchol{}  computes an approximate Cholesky factorization 
\begin{equation} \label{e:chol}
P^\top A P \approx G G^\top,
\end{equation}
where $P$ is a permutation matrix and $G$ is a lower triangular matrix. Using $G G^\top$ as the preconditioner, we can solve \cref{e:axb} with the preconditioned Conjugate Gradient (PCG) method~\cite{saad2003iterative}. Generally, $A$ also has positive off-diagonal entries. In some cases (\Cref{s:bsdd}), we can find a diagonal matrix $D$ with $+1$ or $-1$ on the diagonal such that $DAD$ has only negative nonzero off-diagonal entries; otherwise, we solve an equivalent linear system that has only negative nonzero off-diagonal entries but is twice larger.

\subsection{Related work}


Direct solvers  compute exact factorizations of $A$ and generally require $\bigO(N^3)$ work and $\bigO(N^2)$ storage. Although  matrix $A$ is sparse, a naive direct method may introduce excessive new nonzero entries (a.k.a., fill-in) during the factorization. To minimize fill-in, sparse-matrix reordering schemes such as nested dissection (ND)~\cite{george1973nested} and approximate minimum degree (AMD)~\cite{amestoy1996approximate} are usually employed in state-of-the-art methods, namely, sparse direct solvers~\cite{davis2016survey}. One notable example is the nested-dissection multifrontal method (MF)~\cite{duff1983multifrontal,liu1992multifrontal}, where the elimination ordering and the data flow follow a special hierarchy of \emph{separator fronts}. When applied to matrix $A$ from the discretization of PDEs in three-dimensional space, MF generally reduces the computation and memory complexities to $\bigO(N^2)$ and $\bigO(N^{4/3})$, respectively. However, such costs, dominated by those for factorizing the largest {separator front} of size $\bigO(N^{2/3})$, are still prohibitive for large-scale problems.

Preconditioned iterative methods are often preferred for large scale problems~\cite{saad2003iterative}. A key design decision in iterative solvers is the preconditioner. State-of-the-art methods such as domain decomposition and multigrid methods work efficiently for a large class of problems, including SDD matrices.  A cheaper and simpler alternative is to use an approximate factorization as in \cref{e:chol}, and one popular strategy to compute such a factorization is the incomplete factorization~\cite{meijerink1977iterative}. An incomplete factorization permits fill-in at only specified locations in the resulting factorization. These locations  can be computed in two ways: statically, based on the sparsity structure of $A$ with a level-based strategy; or dynamically, generated during the factorization process with a threshold-based strategy~\cite{saad1994ilut} or its variants~\cite{scott2014hsl_mi28, hook2018max}. Because of its importance, an incomplete Cholesky factorization is often parallelized on single-node shared-memory machines, and this type of parallel algorithm has been  studied extensively~\cite{raghavan2010parallel,chow2015fine,kim2016task,anzt2018parilut}.  {Incomplete factorizations are widely used in computational science and engineering, especially when the underlying physics of a problem is difficult to exploit. Besides being used as a stand-alone preconditioner, an incomplete factorization is also an important algorithmic primitive in more sophisticated methods. For example, it can be used to precondition subdomain solves in domain decomposition schemes or as a smoother in multigrid methods. In this paper, we focus on a randomized scheme for constructing incomplete factorizations. Although we compare our method directly with other solvers, we would like to emphasize that we envision it as an algorithmic primitive in more complex solvers.}

More recently, a class of methods known as the Laplacian Paradigm  have been developed specifically for solving SDD linear systems as in \cref{e:axb}. 
In a breakthrough~\cite{spielman2004nearly}, Spielman and Teng  proved in 2004 that \cref{e:axb} can be solved in nearly-linear time. 
Despite the progress with asymptotically faster and simpler algorithms~\cite{Koutis_2011,kelner2013simple,lee2013efficient,kyng2016approximate}, 
practical implementations of these methods that are able to compete with state-of-the-art linear solvers are limited~\cite{koutis2011combinatorial,livne2012lean}.
A notable recent effort is \texttt{Laplacians.jl},\footnote{\url{https://github.com/danspielman/Laplacians.jl}} a Julia package containing linear solvers for Laplacian matrices, but no results have been reported for solving problems related to PDEs, the target application of our work. 
In this paper, we build on two established ideas: the SparseCholesky algorithm in~\cite{kyng2016approximate}; and a random sampling scheme implemented in \texttt{Laplacians.jl}. In the SparseCholesky algorithm, the Schur-complement update is written as a diagonal matrix plus the graph Laplacian of a clique. Then, edges in the clique are sampled and re-weighted, so the graph Laplacian of sampled edges equals that of the clique in expectation. {In \texttt{Laplacians.jl}, Spielman and Kyng~\cite{sampling} proposed another sampling strategy,} which empirically performed better but has not been analyzed, according to our knowledge and the software documentation.

\subsection{{Contributions}}

In this work, we focus on solving SDD linear systems arising from the discretization of PDEs, and the main ingredient of our approach is an approximate Cholesky factorization constructed via random sampling. In particular, we introduce a randomized Cholesky factorization for  Laplacian matrices {building on top of previous work by Spielman and Kyng~\cite{kyng2016approximate,sampling}}. As observed in~\cite{kyng2016approximate}, eliminating a row/column in the matrix is equivalent to subtracting the graph Laplacian of a \emph{star} and adding the graph Laplacian of a \emph{clique}. {Following~\cite{sampling},  we sample a sparse subset of the edges instead of keeping the full clique.} Our specific contributions  include the following:

\begin{itemize}

\item
  We prove that the sampled edges form a spanning tree on the clique, and consequently, \rchol{} is break-down free for an irreducible Laplacian matrix. We also extend \rchol{} to compute approximate factorizations for subclasses of SDD matrices that are not Laplacian matrices. For the rest of SDD matrices that we cannot apply \rchol{} directly, we clarify how to obtain an approximate solution of \cref{e:axb} under a given tolerance through solving an extended problem using PCG.

\item
We introduce a high-performance parallel algorithm for \rchol{} based on the ND ordering and the multifrontal method. We implemented the parallel algorithm using a task-based approach for shared-memory multi-core machines. Our software offering C++/MATLAB/Python interfaces is available at {\url{https://github.com/ut-padas/rchol}}. 

\item
We benchmarked our code on various problems: Poisson's equation, variable-coefficient Poisson's equation, anisotropic Poisson's equation, and problems from the SuiteSparse Matrix Collection\footnote{\url{https://sparse.tamu.edu/}}. With our benchmark results, we demonstrated the importance of using fill-reducing orderings, the stability and the scalability of our method. We also compared our method to the well-established incomplete Cholesky factorization with threshold dropping. 

\end{itemize}

Our results highlight several features of the new method that are distinct from existing deterministic incomplete Cholesky factorizations: (1) fill-reducing ordering (as opposed to natural/lexicographical ordering) such as AMD and ND improved the performance of our method; (2) the number of iterations required by PCG increased approximately logarithmically with the problem size for discretized 3D Poisson equation; and (3) the performance of our parallel algorithm is hardly affected by the number of threads used.

\subsection{Outline and notations}

The remainder of this paper is organized as follows. \Cref{s:cholesky} introduces \rchol{} with analysis. \Cref{s:sdd} focuses on  solving SDD linear systems and the parallel algorithm for \rchol{}.
\Cref{s:result} presents numerical experiments, and \Cref{s:end} discusses generalizations and draws conclusions.


Throughout this paper, matrices are denoted by capital letters with their entries
given by the corresponding lower case letter in the usual way, e.g., $A = (a_{ij} ) \in \mathbb{R}^{N \times N}$.
We adopt the MATLAB notation to denote a submatrix, e.g., $A(i,:)$ and $A(:,i)$ stand for the $i_{\mathrm{th}}$ row and $i_{\mathrm{th}}$ column in matrix $A$, respectively.


\section{Randomized Cholesky factorization for Laplacian matrix} \label{s:cholesky}

In this section, we focus on irreducible Laplacian matrices, which can be viewed as weighted undirected graphs that have only one connected component. Then, we introduce Cholesky factorization and give the first formal statement of {the clique sampling scheme by Spielman and Kyng~\cite{sampling} in the Laplacians.jl package}. Finally, we provide analysis on the resulting randomized Cholesky factorization.


\begin{definition}[Laplacian matrix~\cite{kyng2016approximate}] \label{d:lm}
Matrix $A \in \mathbb{R}^{N \times N}$ is a Laplacian matrix if (1) $A=A^\top$, (2) $ \sum_{j=1}^N a_{ij} = 0$ for $i=1,2,\ldots,N$, and (3) $a_{ij} \le 0$ when $i\not=j$.
\end{definition}


\begin{definition}[Irreducible matrix~\cite{spielman2014nearly}] \label{d:irreducible}
Matrix $A$ is irreducible if there does not exist a permutation matrix $P$ such that $P^\top A P$ is a block triangular matrix.
\end{definition}

\begin{lemma}[Irreducible Laplacian matrix] \label{l:diag}
Suppose $A \in \mathbb{R}^{N \times N}$ is an irreducible Laplacian matrix. If $N>1$, then $a_{ii} > 0$ for all $i=1,2,\ldots, N$; otherwise $A$ is a scalar zero.
\end{lemma}

Note a Laplacian matrix is always positive semi-definite, and the null space is span\{$\bone$\} if it is irreducible. Below we state a well-known result that there exists a bijection between the class of Laplacian matrices and the class of weighted undirected graphs to prepare for the sampling algorithm.

\begin{definition}[Graph Laplacian] \label{d:gl}
Let $\G=(V,E)$ be a weighted undirected graph, where $V=(v_1,v_2,\ldots,v_N)$, and an edge $e_{ij} = (v_i,v_j) \in E$ carries weight $w_{ij}>0$. The graph Laplacian of $\G$ is
\begin{equation} \label{e:l}
L = \sum_{e_{ij} \in E} w_{ij} \, \bb_{ij} \bb_{ij}^\top,
\end{equation}
where $\bb_{ij}=\be_i-\be_j$, the difference of two standard bases $\be_i,\be_j \in \mathbb{R}^N$ (the order of different does not affect $L$).
\end{definition}

\begin{remark} \label{d:gl2} 
For completeness, we also mention another equivalent definition of graph Laplacian. Given a weighted undirected graph $\G=(V,E)$, the graph Laplacian of $\G$ is
\[
L = D - W,
\]  
where $W$ is the weighted adjacency matrix, i.e., $-w_{ij}$ is the weight associated with edge $e_{ij} \in E$, and $D$ is the weighted degree matrix, i.e., $d_{ii} = - \sum_{j\not= i} w_{ij}$ for all $i$.
\end{remark}

\begin{theorem} \label{t:equivalence}
\cref{d:lm} and \cref{d:gl} are equivalent: matrix $L$ in \cref{e:l} is a Laplacian matrix, and there exists a weighted undirected graph of which the graph Laplacian equals to a given Laplacian matrix.
\end{theorem}

\begin{proof}
Note that
\[
\bordermatrix{
& & i & & j  &\cr
& \ddots &   & & & \cr
i &  &  1 &  \ldots & -1 & \cr
&  & \vdots & & \vdots & \cr
j &  &  -1 &  \ldots & 1 & \cr
&  &   &  &  & \ddots \cr
}
= \bb_{ij} \, \bb_{ij}^\top,
\] 
and it is straightforward to verify that $L$ in \cref{e:l} is a Laplacian matrix. In the other direction, for a given Laplacian matrix $A$, we can construct a weighted undirected graph $\G$ based on the weighted adjacency matrix $D-A$, where $D$ contains the diagonal of $A$. According to \cref{d:gl2}, $A$ is the graph Laplacian of $\G$.  
\end{proof}

\subsection{Cholesky factorization and clique sampling}

Consider applying the Cholesky factorization to an irreducible Laplacian matrix $L \in \mathbb{R}^{N \times N}$ for $N-1$ steps as shown in \cref{a:chol}. It is straightforward to verify that $L$ is always a Laplacian matrix inside the for-loop (line 4). Furthermore, the Schur complement at the $k_\mathrm{th}$ step, i.e., $L(k\text{+1:}N, k\text{+1:}N)$, is an irreducible Laplacian matrix for $k=1,2,\ldots,N-1$. According to \cref{l:diag}, we know that $\ell_{kk} > 0$ at line 3 and $\ell_{NN}=0$ after the for-loop. An irreducible Laplacian matrix corresponds to a connected graph, and the zero Schur complement, which stands for an isolated vertex, would not occur earlier until the other $N-1$ vertices have been eliminated.



\begin{algorithm}
\caption{Classical Cholesky factorization for Laplacian matrix}
\label{a:chol}
\begin{algorithmic}[1]
\Require irreducible Laplacian matrix $L \in \mathbb{R}^{N \times N}$
\Ensure lower triangular matrix $G \in \mathbb{R}^{N \times N}$
\State $G = \bm{0}_{N \times N}$
\For {$k=1$ \textbf{to} $N-1$}
\State $G(:,k) = L(:,k)/\sqrt{\ell_{kk}}$
\hfill \Comment \cyan{// $\ell_{kk} > 0$ for an irreducible Laplacian input}
\State $L = L - \frac{1}{\ell_{kk}} \, L(:, k) \, L(k, :)$
\hfill \Comment \cyan{// dense Schur-complement update}
\EndFor
\end{algorithmic}
\end{algorithm}

At the $k_\mathrm{th}$ step in \cref{a:chol}, the elimination (line 4) leads to a dense sub-matrix in the Schur complement. Next, we use the idea of random sampling to reduce the amount of fill-in. At the $k_\mathrm{th}$ step, we define the neighbors of $k$ as 
\begin{equation} \label{e:nbor}
\mathcal{N}_k \triangleq \{i: \ell_{ki}\not=0, i\not=k\},
\end{equation}
corresponding to vertices connected to vertex $k$ in the underlying graph. We also define the graph Laplacian of the sub-graph consisting of $k$ and its neighbors as
\begin{equation}
L^{(k)} \triangleq \sum_{i \in \mathcal{N}_k} \, (-\ell_{ki}) \,\, \bb_{ki} \bb_{ki}^\top
\end{equation}
It is observed in~\cite{kyng2016approximate} that the elimination at line 4 in \cref{a:chol} can be written as the sum of two Laplacian matrices:
\[
L - \frac{1}{\ell_{kk}} \, L(:, k) \, L(k, :) 
=
\underbrace{L - L^{(k)}}_{\text{Laplacian matrix}} +
\underbrace{L^{(k)} - \frac{1}{\ell_{kk}}L(:, k)L(k, :)}_{\text{Laplacian matrix}}
\]
The first term is the graph Laplacian of the sub-graph consisting of all edges except the ones connected to $k$. Since  
\[
L(:,k)  - L^{(k)}(:,k) = 0, \quad L(k,:) - L^{(k)}(k,:) = 0,
\]
we know $L - L^{(k)}$ zeros out the $k_\mathrm{th}$ row/column in $L$ and updates the diagonal entries in $L$ corresponding to $\mathcal{N}_k$.

The second term
\begin{equation} \label{e:clique}
    L^{(k)} - \frac{1}{\ell_{kk}}L(:, k)L(k, :) = 
    \frac{1}{2} \sum_{i,j \in \mathcal{N}_k} \frac{\ell_{ki} \, \ell_{kj}}{\ell_{kk}} \, \bb_{ij} \bb_{ij}^\top
\end{equation}
is the graph Laplacian of the clique among neighbors of $k$, where the edge between neighbor $i$ and neighbor $j$ carries weight ${\ell_{ki} \, \ell_{kj}}/{\ell_{kk}}$. Denote the number of neighbors of $k$ as $n$, i.e., 
\[
n \triangleq |\mathcal{N}_k|
\]
 Note \cref{e:clique} is a dense matrix with $n^2$ entries or a clique with $\bigO(n^2)$ edges. The idea of randomized Cholesky factorization is to sample $\bigO(n)$ edges from the clique (and assign new weights), corresponding to $\bigO(n)$ fill-in entries. The randomized algorithm is shown in \cref{a:rchol}, and the difference from \cref{a:chol} is shown pictorially with an example in \cref{f:chol}.

 

\begin{algorithm}
\caption{Randomized Cholesky factorization for Laplacian matrix}
\label{a:rchol}
\begin{algorithmic}[1]
\Require irreducible Laplacian matrix $L \in \mathbb{R}^{N \times N}$
\Ensure lower triangular matrix $G \in \mathbb{R}^{N \times N}$ 
\State $G = \bm{0}_{N \times N}$
\For {$k=1$ \textbf{to} $N-1$}
\State $G(:,k) = L(:,k)/\sqrt{\ell_{kk}}$
\hfill \Comment \cyan{// $\ell_{kk} > 0$ according to \cref{th:stable}}
\State $L = L - L^{(k)} + \Call{SampleClique}{L, k}$
\hfill \Comment \cyan{// sparse Schur-complement update}
\EndFor
\end{algorithmic}
\end{algorithm}

\begin{figure}
\hfill
\begin{subfigure}{.25\textwidth}
  \centering
  \includegraphics[width=.8\linewidth]{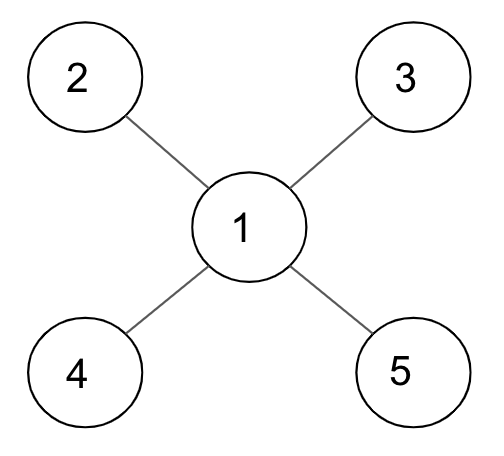}  
\end{subfigure}
\hfill
\begin{subfigure}{.25\textwidth}
  \centering
  \includegraphics[width=.8\linewidth]{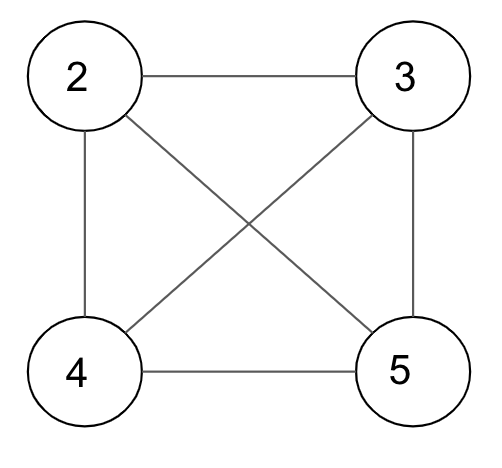}  
\end{subfigure}
\hfill
\begin{subfigure}{.25\textwidth}
  \centering
  \includegraphics[width=.8\linewidth]{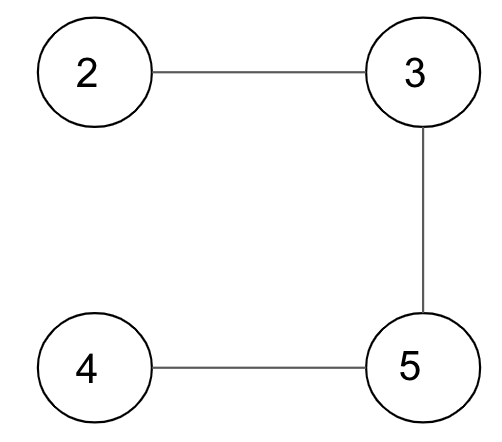}  
\end{subfigure}
\hfill
\caption{An example: (left) graph of $L$ before vertex 1 is eliminated; (middle) graph of the Schur complement after vertex 1 is eliminated; and (right) a randomly sampled subset of the clique.}
\label{f:chol}
\begin{center}
\scalebox{0.24}{\includegraphics{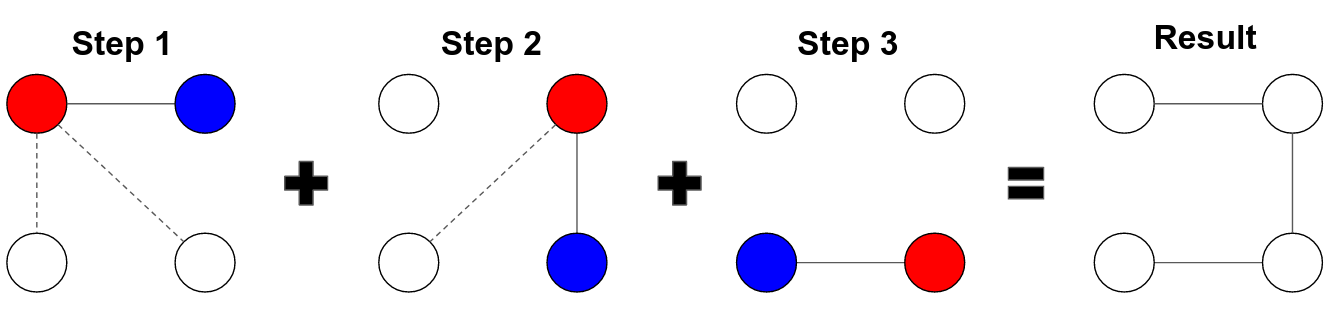}}
\caption{An instance of \cref{a:sample} for the example in \cref{f:chol}. At every step, the red vertex stands for $i \in \mathcal{N}$ at Line 5 in \cref{a:sample}; the  blue vertex stands for $j \in \mathcal{N}$ at Line 8; the solid line is the sampled edge; and the dashed lines are other potential candidates for sampling.
}
\label{f:sample}
\end{center}
\end{figure}


The pseudocode of the sampling algorithm is shown in \cref{a:sample}, which selects $n-1$ edges from a clique among $n$ vertices as follows. Before sampling, the neighbors of $k$ are sorted in ascending order based on their weights $|\ell_{ki}|$. For every $i \in \mathcal{N}_k$, we sample $j \in \mathcal{N}_k$ such that $|\ell_{kj}| > |\ell_{ki}|$ with a probability proportional to $|\ell_{kj}|$. Then, an edge between $i$ and $j$ is created with an appropriate weight (so the graph Laplacian of the sampled edges equals to \cref{e:clique} in expectation; see \cref{th:sample}). \cref{f:sample} shows an example of the sampling process step-by-step.

\begin{algorithm}
\caption{Sample clique ({by Spielman and Kyng~\cite{sampling}})}
\label{a:sample}
\begin{algorithmic}[1]
\Require Laplacian matrix $L \in \mathbb{R}^{N \times N}$ and elimination index $k$
\Ensure graph Laplacian of sampled edges $C \in \mathbb{R}^{N \times N}$ 
\State $C = \bm{0}_{N \times N}$
\State Sort $\mathcal{N}_k$ in ascending order based on $|\ell_{ki}|$ for $i \in \mathcal{N}_k$
\hfill \Comment{\cyan{// $\mathcal{N}_k$ defined in Eq. \cref{e:nbor}}}
\State $S=\ell_{kk}$
\hfill \Comment{\cyan{// $\ell_{kk} = - \sum_{i \in \mathcal{N}_k} \ell_{ki}$}}
\While{$ \left| \mathcal{N}_k \right| > 1$}
\State Let $i$ be the first element in $\mathcal{N}_k$
\hfill \Comment{\cyan{// loop over neighbors}}
\State{$\mathcal{N}_k = \mathcal{N}_k / \{i\}$} 
\hfill \Comment{\cyan{// remove $i$ from the set}}
\State $S=S+\ell_{ki}$
\hfill \Comment{\cyan{// $S = - \sum_{j \in \mathcal{N}_k} \ell_{kj}$}}
\State{Sample  $j$ from $\mathcal{N}_k$ with {probability $|\ell_{kj}|/S$}} 
\State $C = C - \frac{S\, \ell_{ki}}{\ell_{kk}} \, \bb_{ij} \bb_{ij}^\top$
\hfill \Comment{\cyan{// pick edge $(i,j)$; assign weight ${S\, |\ell_{ki}|}/{\ell_{kk}}$}}
\EndWhile
\end{algorithmic}
\end{algorithm}

\subsection{Analysis of randomized Cholesky factorization} \label{s:analysis}


In this section, we prove the robustness and the scalability of \rchol{}. The following theorem shows that the  edges sampled by \cref{a:sample} form a spanning tree, and consequently,  \cref{a:rchol} never breaks down.

\begin{theorem}[Spanning tree on clique] \label{th:connect}
The sampled edges in \cref{a:sample} form a spanning tree of the clique on neighbors of $k$.
\end{theorem}
\begin{proof}
Suppose $k$ has $n$ neighbors. Observe $n-1$ edges are sampled and all neighbors are included in the graph formed by these edges. It remains to be proved that this graph is connected.
 
Suppose the set of neighbors $\mathcal{N}_k$ is sorted in ascending order. We can find a path between any $i \in \mathcal{N}_k$ and the last/``heaviest'' element in $\mathcal{N}_k$ with the following rational:
\begin{enumerate}
\item
Start from any $i \in \mathcal{N}_k$. Suppose a sampled edge goes from $i$ to a ``heavier'' neighbor $j \in \mathcal{N}_k$ ($|\ell_{ki}| < |\ell_{kj}|$).
\item
Move to $j$, and repeat the previous process. It follows that we will reach the ``heaviest'' neighbor after a finite number of steps.
\end{enumerate}

\end{proof}



\begin{corollary}[Breakdown free] \label{th:stable}
In \cref{a:rchol}, $\ell_{kk} > 0$ at line 3 and $\ell_{\text{NN}}=0$ after the for-loop.
\end{corollary}
\begin{proof}
Since \cref{a:sample} returns a graph Laplacian of a connected graph among the neighbors of $k$ at line 4 in \cref{a:rchol}, it is straightforward to verify that the Schur complement at the $k_\mathrm{th}$ step (i.e., $L(k\text{+1:}N, k\text{+1:}N)$) is an irreducible Laplacian matrix. Therefore, this corollary holds according to \cref{l:diag}.
\end{proof}


{The next theorem addresses the time complexity and the storage of \rchol{} employing a random elimination ordering, which follows the argument in~\cite{kyng2016approximate} closely. (We prove this in \cref{s:a}).} 

\begin{theorem}[Running time and storage] \label{th:complexity}
Suppose an irreducible Laplacian matrix $L\in \mathbb{R}^{N \times N}$ has $M$ non-zeros, and a random row/column is eliminated at every step in \cref{a:rchol}.
Then, the expected running time of \cref{a:rchol} is upper bounded by $\bigO(M \log N)$, and the expected number of non-zeros in the output triangular matrix $G$ is upper bounded by $\bigO(M \log N)$.
\end{theorem}

The next theorem shows that \cref{a:sample} returns an unbiased estimator at every step in \cref{a:rchol}.



\begin{theorem}[Unbiased estimator] \label{th:sample}
At the $k_\mathrm{th}$ step in \cref{a:rchol}, the expectation of $C = \text{SampleClique(L,k)}$ equals to the result of exact elimination, as defined in \cref{e:clique}. 
\end{theorem}
\begin{proof}
Suppose $i,j \in \mathcal{N}_k$ and $0 < |\ell_{ki}| < |\ell_{kj}|$. The probability that edge $(i,j)$ being sampled is 
$
P_{ij} = {|\ell_{kj}|}/{S},
$
according to Line 8 in \cref{a:sample}. Therefore, we have
\begin{align*}
\mathbb{E}[C] 
&= \sum_{i,j \in \mathcal{N}_k \text{ and } |\ell_{ki}| < |\ell_{kj}|} P_{ij} \frac{ S\, (-\ell_{ki})}{\ell_{kk}} \, \bb_{ij} \bb_{ij}^\top
=  \sum_{i,j \in \mathcal{N}_k \text{ and } |\ell_{ki}| < |\ell_{kj}|} \frac{\ell_{kj} \, \ell_{ki}}{\ell_{kk}} \, \bb_{ij} \bb_{ij}^\top 
\end{align*}
\end{proof}

\subsection{Relation to approximate Cholesky factorizations in~\cite{kyng2016approximate} and~\cite{sampling}}

While both \rchol{} and  the method in~\cite{kyng2016approximate} follow the same template of \cref{a:rchol}, they differ in two manners. The first difference is that the algorithms of clique sampling  are different. In~\cite{kyng2016approximate} the authors propose to sample $n$ edges from a clique at every step in \cref{a:chol}. To sample an edge, a neighbor $i$ is sampled uniformly from $\mathcal{N}_k$, and a neighbor $j$ is sampled from $\mathcal{N}_k$ with probability $|\ell_{kj}|/\ell_{kk}$; then, an edge between $i$ and $j$ is created with weight $\ell_{ki}\ell_{kj}/|\ell_{ki}+\ell_{kj}|$ if $i \not= j$. With such a sampling strategy, an edge can be sampled repeatedly, and there is a probability that no edge is created (when $i$ and $j$ are identical). So \cref{a:sample} can be viewed as a derandomized variant of the sampling in~\cite{kyng2016approximate}. 

The other difference is that there is an extra initialization step before entering \cref{a:rchol} in~\cite{kyng2016approximate}. For a Laplacian matrix, the initialization is to split every edge in the associated graph into $\rho = \bigO(\log^2 N)$ copies with $1/\rho$ of the original weight. Then,  the resulting \emph{multi-graph} becomes the input of \cref{a:rchol}. It was proven that the norm of the \emph{normalized graph Laplacian} associated with every edge in the multi-graph is upper bounded by $1/\rho$ throughout the factorization with the aforementioned sampling algorithm. As a result, a nearly-linear time solver was obtained as  the following theorem states.

\begin{theorem}[Approximate Cholesky factorization in~\cite{kyng2016approximate}]
Let $L \in \mathbb{R}^{N \times N}$ be an irreducible Laplacian matrix with $M$ non-zeros, and $P \in \mathbb{R}^{N \times N}$ be a random permutation matrix. If we perform the above initialization step on $P^\top L P$ and apply  \cref{a:rchol} with the above sampling algorithm, then the expected running time is $\bigO(\rho M \log N) = \bigO(M \log^3 N)$, and the expected number of non-zeros in the output triangular matrix $G$ is $\bigO(\rho M \log N) = \bigO(M \log^3 N)$. In addition, with high probability,
\[
\frac{1}{2}  L  \preceq {(PG)} {(PG)}^\top \preceq \frac{3}{2}  L .
\]
(For two symmetric matrices $A$ and $B$, the notation $A \preceq B$ means that $B-A$ is a positive semi-definite matrix.)
\end{theorem}

Overall, the algorithm in~\cite{kyng2016approximate} requires a more expensive factorization than \rchol{} (the extra $\log^2 N$ factor in the running time can be significant in practice), but it produces an approximation of better quality.


{Compared to~\cite{sampling}, \rchol{} computes a mathematically equivalent operator if the same elimination ordering is used. (\rchol{} by default uses the AMD ordering~\cite{amestoy1996approximate} in practice; see~\Cref{s:r1}.) Hence, our analysis for \rchol{} also applies to the method in~\cite{sampling}. While \rchol{} represents the output as an approximate Cholesky factorization, \cite{sampling} uses a row-operation representation. 
}

\section{Randomized preconditioner for SDD matrix} \label{s:sdd}

In this section, we consider an SDD linear system $Ax=b$, where $A$ is irreducible as defined in \cref{d:irreducible} but not a Laplacian matrix.
In \Cref{s:sddm}, we consider the case when $A$ is an SDDM matrix, which can be viewed as the sum of a Laplacian matrix and a non-negative diagonal matrix with at least one positive diagonal entry. In \Cref{s:bsdd}, we introduce bipartite SDD matrices, a subclass of SDD matrices containing positive off-diagonal entries but  can be converted to either a Laplacian matrix or an SDDM matrix through diagonal scaling.

When $A$ is either an SDDM matrix or a bipartite SDD matrix, we can compute an approximate Cholesky factorization of $A$ and use it as a preconditioner to solve for $x$. Otherwise, it is well-known in the literature~\cite{gremban1996combinatorial} that $x$ can be obtained through solving a twice larger linear system $\tilde{A} y = \tilde{b}$ \emph{in exact arithmetic}. In \Cref{s:sdd2}, we show how to retrieve an approximate solution $x$ that has the same relative residual as a given \emph{approximate solution} $y$ for the larger system.

\subsection{SDDM matrix} \label{s:sddm}

\begin{definition} 
Matrix $A  \in \mathbb{R}^{N \times N}$ is a {s}ymmetric {d}iagonally {d}ominant {M}-matrix if $A$ is (1) SDD, (2) positive definite, and (3) $a_{ij} \le 0$ when $i\not=j$.
\end{definition}

Our goal is to compute an approximate Cholesky factorization for an SDDM matrix $A$:
\begin{equation} \label{e:agg}
 A \approx G G^\top.
\end{equation}
The factorization can be used as  a preconditioner for solving $Ax=b$. To obtain \cref{e:agg}, our approach is  applying \cref{a:rchol} to the following extended matrix that initially appread in~\cite{gremban1996combinatorial}:
\begin{equation} \label{e:sddme}
\tilde{A}
\triangleq
\begin{pmatrix}
A & -A \bone \\
-\bone^\top A & \bone^\top A \bone
\end{pmatrix},
\quad \tilde{A} \in \mathbb{R}^{(N+1) \times (N+1)}
\end{equation}
where $\bone  \in \mathbb{R}^{N}$ stands for the all-ones vector. The reason we can apply \cref{a:rchol} is the following lemma.

\begin{lemma} \label{l:sddm}
Given an irreducible SDDM matrix $A$, the extended matrix $\tilde{A}$, defined in \cref{e:sddme}, is an irreducible Laplacian matrix.
\end{lemma}
\begin{proof}
Since $A$ is SDD and positive definite, the row-sum vector $A \bone$ has non-negative entries and at least one positive entry. Therefore, it is straightforward to verify that $\tilde{A}$ is an irreducible Laplacian matrix.
\end{proof}

Suppose the output of \cref{a:rchol} is the following:
\begin{equation} \label{e:gt}
\rchol(\tilde{A})
\triangleq \tilde{G}
=
\begin{pmatrix}
\tilde{G}_{11} &  \\
\tilde{G}_{21} & \tilde{g}_{22}
\end{pmatrix}
,
\end{equation}
where $\tilde{G}_{11} \in \mathbb{R}^{N \times N},  \tilde{G}_{21} \in \mathbb{R}^{1 \times N}$ and $\tilde{g}_{22} \in \mathbb{R}$. We know that $\tilde{g}_{22}=0$ according to \cref{th:stable}. In other words, we have the following approximation:
\[
\tilde{A}
=
\begin{pmatrix}
A & -A \bone \\
-\bone^\top A & \bone^\top A \bone
\end{pmatrix}
\approx
\tilde{G} \tilde{G}^\top
=
\begin{pmatrix}
\tilde{G}_{11}  &  \\
\tilde{G}_{21}  & 0
\end{pmatrix}
\begin{pmatrix}
\tilde{G}_{11} ^\top & \tilde{G}_{21} ^\top \\
 & 0
\end{pmatrix}
,
\]
from which we see that 
\begin{equation*} \label{e:agg1}
A \approx \tilde{G}_{11}  \tilde{G}_{11} ^\top
\end{equation*}
in the leading principle block. We summarize the above algorithm in \cref{a:sddm}.

\begin{algorithm}
\caption{Randomized Cholesky factorization for SDDM matrix}
\label{a:sddm}
\begin{algorithmic}[1]
\Require irreducible SDDM matrix $A \in \mathbb{R}^{N \times N}$
\Ensure lower triangular matrix $G \in \mathbb{R}^{N \times N}$ 
\State Construct $\tilde{A}$ defined in \cref{e:sddme}.
\State Compute
\[
\begin{pmatrix}
\tilde{G}_{11} &  \\
\tilde{G}_{21} & 0
\end{pmatrix}
=
\Call{RandomizedCholesky}{\tilde{A}}
\quad\quad {\text{\cyan{// call \cref{a:rchol}}}}
\]
where $\tilde{G}_{11} \in \mathbb{R}^{N \times N}$ and $\tilde{G}_{21} \in \mathbb{R}^{1 \times N}$.
\State \Return $G = \tilde{G}_{11}$.
\end{algorithmic}
\end{algorithm}

\begin{remark}[Reducible SDDM matrix] \label{r:sddm}
In general, \cref{a:sddm} can be applied to an SDDM matrix $A$ that is reducible because \cref{e:sddme} is still an irreducible Laplacian matrix. However, it may be more efficient to apply \cref{a:sddm} to each irreducible component for solving a linear system with $A$.
\end{remark}

Before ending this section, we justify using $\tilde{G}_{11}  \tilde{G}_{11}^\top$ as a preconditioner through the following classical result.

\begin{theorem}[{{\cite[Lemma 4.2, page 56]{gremban1996combinatorial}}}]
 \label{t:sddm}
Solving an irreducible SDDM linear system $Ax=b$ is equivalent to solving the following irreducible Laplacian linear system
\begin{equation} \label{e:ext1}
\tilde{A} \, y
=
\begin{pmatrix}
b \\
-\bone^\top b
\end{pmatrix}
\end{equation}
\end{theorem}
\begin{proof}
It can be verified the solution of \cref{e:ext1} is
\begin{equation} \label{e:solve}
y = 
\begin{pmatrix}
x \\
0
\end{pmatrix}
+
\text{span}\{\bone\}
.
\end{equation}
Therefore, we can solve \cref{e:ext1} to obtain $x$ and vice versa.
\end{proof}

To solve \cref{e:ext1} and obtain $x$, we first apply PCG with the preconditioner $\tilde{G} \tilde{G}^\top$ in \cref{e:gt}. Then, we orthogonalize the PCG solution with respect to $\text{span}\{\bone\}$. This process turns out to be equivalent to using $\tilde{G}_{11} \tilde{G}_{11}^\top$ as the preconditioner (note $\tilde{G}_{11}$ is non-singular) for solving $Ax=b$ with PCG directly, without going through the extended problem.

\subsection{SDD matrix} \label{sec:sdd}

Given an irreducible SDD matrix $A \in \mathbb{R}^{N \times N}$, let 
\[
A \triangleq A_d + A_n + A_p,
\]
where $A_d, A_n, A_p  \in \mathbb{R}^{N \times N}$ contain the diagonal, the negative off-diagonal and the positive off-diagonal entries of $A$, respectively. 
In this section, we focus on the case when $A_p \not  = \bm{0}$, i.e., $A$ contains at least two positive off-diagonal entries (due to symmetry).

\subsubsection{{Bipartite SDD matrix}} \label{s:bsdd}

We introduce bipartite SDD matrices and give three equivalent definitions below (proof is in \cref{s:b}).

\begin{definition} \label{d:bsdd} 
A bipartite SDD matrix $A$ can be defined in any of the following three equivalent ways:
\begin{enumerate} [label=(\alph*)]
\item 
Let $\hat{A}$ be an SDD matrix defined by the off-diagonal part of $A$:
\begin{equation} \label{e:hat}
\hat{A} \triangleq \diag \left( (A_p - A_n) \bone \right) + A_p + A_n,
\end{equation}
where $\diag(\cdot)$ maps a vector to a diagonal matrix. 
If $\mathrm{rank}(\hat{A}) = N-1$, then $A$ is a bipartite SDD matrix.

\item
Let $D$ be  a diagonal matrix, whose diagonal entries are ether 1 or -1. If there exists such a matrix $D$ that $DAD$ has only non-positive off-diagonal entries, then $A$ is a bipartite SDD matrix.

\item
Let $\G=(V,E)$ be a undirected graph, where $V=(v_1,v_2,\ldots,v_N)$ has $N$ vertices;  an edge $e_{ij} = (v_i,v_j) \in E$ exists if $a_{ij} \not = 0$ and carries weight $w_{ij} = - a_{ij}$.  If the graph $\G$ is 2-colorable (bipartite) in the following sense:
\begin{itemize}
\item
$v_i$ and $v_j$ have the same color if $w_{ij} > 0$;

\item
 $v_i$ and $v_j$ have different colors if $w_{ij} < 0$,
\end{itemize}
then $A$ is a bipartite SDD matrix.

\end{enumerate}
\end{definition}

\begin{example}
The following shows three $3\times 3$ SDD matrices with positive off-diagonal entries, where a symbol $\times$ denotes any value greater than or equal to 2. Among the three matrices, $A_1$ is a bipartite SDD matrix and the other two are not.
\[
A_1 =
\begin{pmatrix}
\times & 1 & 1  \\
1 & \times & -1  \\
1 & -1 & \times  \\
\end{pmatrix}
\quad
A_2 =
\begin{pmatrix}
\times & 1 & -1  \\
1 & \times & -1  \\
-1 & -1 & \times  \\
\end{pmatrix}
\quad
A_3 =
\begin{pmatrix}
\times & 1 & 1  \\
1 & \times & 1  \\
1 & 1 & \times  \\
\end{pmatrix}
\]
\end{example}

\begin{remark} \label{r:bsdd}
Whether $A$ is a bipartite SDD matrix or not depends on only its off-diagonal part according to \cref{d:bsdd} (a). When $A_p \not = \bm{0}$, we have $\mathrm{rank}(\hat{A}) = N$ if $A$ is not a bipartite SDD matrix. Otherwise, when $A_p = \bm{0}$ ($A$ is either a Laplacian matrix or an SDDM matrix), we have $\mathrm{rank}(\hat{A}) = N-1$.
\end{remark}

Our goal is to compute an approximate (generalized) Cholesky factorization of an irreducible bipartite SDD matrix. In the following, we show that it takes linear time to find the matrix $D$ in \cref{d:bsdd} (b), and thus we can apply \rchol{} to $DAD$, which is either a Laplacian matrix or an SDDM matrix. Given an irreducible SDD matrix, \cref{a:bipartite} tries to find the matrix $D$ by traversing the graph $\G$ defined in \cref{d:bsdd} (c). \cref{a:bipartite} is based on the breadth-first-search  and can also be implemented in the depth-first-search. With the matrix $D$, we obtain an approximate (generalized) Cholesky factorization $A \approx G G^\top$, where $G$ has both positive and negative diagonal entries.

\begin{algorithm}
\caption{Check bipartite SDD matrix}
\label{a:bipartite}
\begin{algorithmic}[1]
\Require irreducible  SDD matrix $A \in \mathbb{R}^{N \times N}$ (not necessarily bipartite)
\Ensure flag \texttt{BSDD\_or\_not} and  diagonal matrix  $D \in \mathbb{R}^{N \times N}$ (if $A$ is bipartite)

\State Let \texttt{BSDD\_or\_not} = \texttt{true} and  $d_{11} = 1$.
\State Mark index 1 as visited; and \texttt{queue}.push(1).

\While {\texttt{queue} is not empty}
\State $i$ = \texttt{queue}.pop()
\For {$k: a_{ik} \not = 0, k \not = i$}
\If {index $k$ has not been visited}
    \If {$a_{ik} < 0$}
    \State Let $d_{kk} = d_{ii}$.
    \Else
    \State Let $d_{kk} = -d_{ii}$.
    \EndIf
    \State Mark index $k$ as visited; and \texttt{queue}.push($k$).
\Else
    \If { $a_{ik} \, d_{kk} \, d_{ii} > 0$ } \hfill \Comment{\cyan{// see lines 7-11}}
    \State Let \texttt{BSDD\_or\_not} = \texttt{false} and \Return.
    \EndIf
\EndIf
\EndFor
\EndWhile
\end{algorithmic}
\end{algorithm}

\begin{algorithm}
\caption{Randomized  Cholesky factorization for bipartite SDD matrix}
\label{a:opt}
\begin{algorithmic}[1]
\Require irreducible bipartite SDD matrix $A$
\Ensure lower triangular matrix $G$ 

\State $D = \Call{CheckBipartiteSDDMatrix}{A}$
\State $\tilde{G} = \Call{RandomizedCholesky}{DAD}$
\hfill \Comment{\cyan{// \cref{a:rchol} or \cref{a:sddm}}}

\State $G = D \tilde{G}$
\hfill \Comment{\cyan{// $A \approx D \tilde{G} \tilde{G}^\top D$}}

\end{algorithmic}
\end{algorithm}

\subsubsection{General SDD matrix} \label{s:sdd2}

We consider solving $Ax=b$, where $A_p \not = \bm{0}$ and $A$ is not a bipartite SDD matrix ($A$ is non-singular according to \cref{r:bsdd}). Our goal is to find $x$ such that the relative residual is smaller than a prescribed tolerance $\epsilon$, i.e.,
\begin{equation} \label{e:x}
\|b - Ax\|/\| b \| < \epsilon,
\end{equation}
a common stopping criteria for iterative solvers such as PCG. 
Our approach is to solve the extended system $\tilde{A} y = \tilde{b}$ as initially proposed in~\cite{gremban1996combinatorial}, where
\begin{equation} \label{e:sdde}
\tilde{A}
\triangleq
\begin{pmatrix}
A_d+A_n & -A_p \\
-A_p & A_d+A_n
\end{pmatrix}
, \quad
\tilde{b} 
\triangleq 
\begin{pmatrix}
b \\
-b
\end{pmatrix},
\end{equation}
and we seek to find $y$ satisfying
\begin{equation} \label{e:ayb}
\| \tilde{b} - \tilde{A} y \|/\| \tilde{b} \| < \epsilon.
\end{equation}
Before discussing how to solve the extended system,  we state our main result in the following theorem.
\begin{theorem} \label{th:sdd}
Given 
$
y = 
\begin{pmatrix}
y_1 \\
-y_2
\end{pmatrix}
$
such that \cref{e:ayb} holds, 
where
$
y_1, y_2 \in \mathbb{R}^N.
$
The vector
\begin{equation} \label{e:exact}
x = \frac{y_1 + y_2}{2}
\end{equation}
satisfies \cref{e:x}.
\end{theorem}
\begin{proof}
According to \cref{e:ayb}, we have
\begin{align*}
\| \tilde{b} - \tilde{A} y \|^2
&=
\left\| 
\begin{pmatrix}
b-(A_d+A_n)y_1-A_py_2 \\
b-A_py_1-(A_d+A_n)y_2
\end{pmatrix}
\right\|^2
\\
&= 
\left\|  
b-(A_d+A_n)y_1-A_py_2
\right\|^2
+ 
\left\|  
b-A_py_1-(A_d+A_n)y_2
\right\|^2
\\
&< \epsilon^2 \| \tilde{b} \|^2, 
\end{align*}
where $\| \tilde{b} \|^2 = 2 \| b \|^2$. We obtain \cref{e:x} as follows:
\begin{align*}
\left\|
b-Ax
\right\|^2
&=
\frac{1}{4}
\|
2b - (A_d+A_n+A_p)(y_1+y_2)
\|^2
\\
&=
\frac{1}{4}
\|
\left( b-(A_d+A_n)y_1-A_py_2 \right)
+
\left( b-A_py_1-(A_d+A_n)y_2 \right)
\|^2
\\
& \le \frac{1}{2} \|
\left( b-(A_d+A_n)y_1-A_py_2 \right)
\|^2
+
\frac{1}{2} \| 
\left( b-A_py_1-(A_d+A_n)y_2 \right)
\|^2 \\
&< 
\epsilon^2 \| b \|^2.
\end{align*}
\end{proof}

A similar result on the relative errors also holds~\cite{spielman2014nearly}: 
\[
\left\| y - \tilde{A}^\dagger \tilde{b} \right\| \le  \epsilon \left\| \tilde{A}^\dagger \tilde{b} \right\| \quad \text{implies} \quad \left\| x - {A}^{-1} {b} \right\| \le  \epsilon \left\| {A}^{-1} {b} \right\|,
\]
where $\tilde{A}^\dagger$ denotes the pseudo-inverse of $\tilde{A}$. ($\tilde{A}$ may be singular, i.e., a Laplacian matrix.) In addition, if we seek for the exact solution, i.e., $\epsilon = 0$, then \cref{e:exact} is indeed the solution of $Ax=b$~\cite{maggs2005finding,spielman2014nearly}.

Next, we focus on solving the extended system $\tilde{A} y = \tilde{b}$. It is easy to see that 
$\tilde{A}$ is an SDD matrix with non-positive off-diagonal entries, i.e., a Laplacian matrix or an SDDM matrix.
In addition, $\tilde{A}$ is irreducible as the following theorem states (proof is in \cref{s:c}):
\begin{theorem} \label{th:irreducible}
If an irreducible SDD matrix $A$ contains positive off-diagonal entries ($A_p \not = \bm{0}$)  and is not a bipartite SDD matrix, then the matrix $\tilde{A}$ defined in \cref{e:sdde} is irreducible.
\end{theorem}


Therefore, we can construct an approximate Cholesky factorization of $\tilde{A}$, solve the extended system with PCG and obtain $x$ according to \cref{th:sdd}.
 To summarize, \cref{a:sdd} shows the pseudocode of solving a general irreducible SDD linear system. 

\begin{algorithm}
\caption{General SDD linear solver}
\label{a:sdd}
\begin{algorithmic}[1]
\Require irreducible SDD matrix $A \in \mathbb{R}^{N \times N}$, right-hand side $b \in \mathbb{R}^{N}$ and tolerance $\epsilon$
\Ensure $x \in \mathbb{R}^{N}$ satisfying \cref{e:x}.
\State Construct $\tilde{A}$ and $\tilde{b}$ as defined in \cref{e:sdde}.

\State Compute
\[
\tilde{G} = \Call{RandomizedCholesky}{\tilde{A}}.
\]
\hfill \Comment{\cyan{// \cref{a:rchol} or \cref{a:sddm}}}

\State Compute
\[
\begin{pmatrix}
x_1 \\
-x_2
\end{pmatrix}
= 
\Call{PCG}{\tilde{A}, \, \tilde{b}, \, \epsilon, \tilde{G}, \tilde{G}^\top},
\quad x_1,x_2 \in \mathbb{R}^{N}.
\]
\hfill \Comment{\cyan{// PCG with preconditioner $\tilde{G} \tilde{G}^\top$}}

\State \Return $x = (x_1+x_2)/2$.
\end{algorithmic}
\end{algorithm}

\section{Sparse matrix reordering and parallel algorithm} \label{s:parallel}

In this section, we discuss two techniques for improving the practical performance of \cref{a:rchol} including reordering the input sparse matrix and parallelizing the computation.

Sparse matrix reordering is a mature technique that is used in sparse direct solvers to speed up factorization and to reduce the memory footprint. {Since \cref{a:rchol} keeps a subset of fill-in at every step, it is intuitive that \cref{a:rchol} can also benefit from an appropriate ordering. The challenge, however, is that the fill-in pattern as a result of the random sampling algorithm is not deterministic and thus is impossible to predict beforehand. We resort to using the approximate minimum degree (AMD) ordering~\cite{amestoy1996approximate}, a fill-in reducing heuristic for the  (exact) Cholesky factorization. The advantage is that the AMD can be precomputed quickly and applied to the input sparse matrix before \cref{a:rchol}. In practice, we find the AMD working well with \rchol{}, although the fill-in behavior of \cref{a:rchol} is quite different from that of the  (exact) Cholesky factorization. We present comparisons between the AMD and other popular reordering strategies used in sparse direct solvers in  \Cref{s:r1} .}

Next, we introduce a parallel algorithm for \cref{a:rchol} based on the nested dissection scheme~\cite{george1973nested}. Consider the underlying graph associated with  a given sparse matrix. If we split it into two disconnected components separated by a vertex separator, then we can apply \cref{a:rchol} on the two disconnected pieces using two threads in parallel. When more than two threads are available, we apply the same partitioning recursively on the two independent partitions to obtain more disconnected parts of the graph; see \cref{f:nd} (left) for a pictorial illustration. Technically, the above procedure is known as the nested dissection and can be computed algebraically using METIS/ParMETIS~\cite{georgeMetis,karypis1997parmetis}. Moreover, we employ the AMD ordering within each independent region at the leaf level. The pseudocode of our ordering strategy is shown in \cref{a:order}, which can be parallelized in a straightforward way.

\begin{figure}
\begin{center}
\scalebox{0.22}{\includegraphics{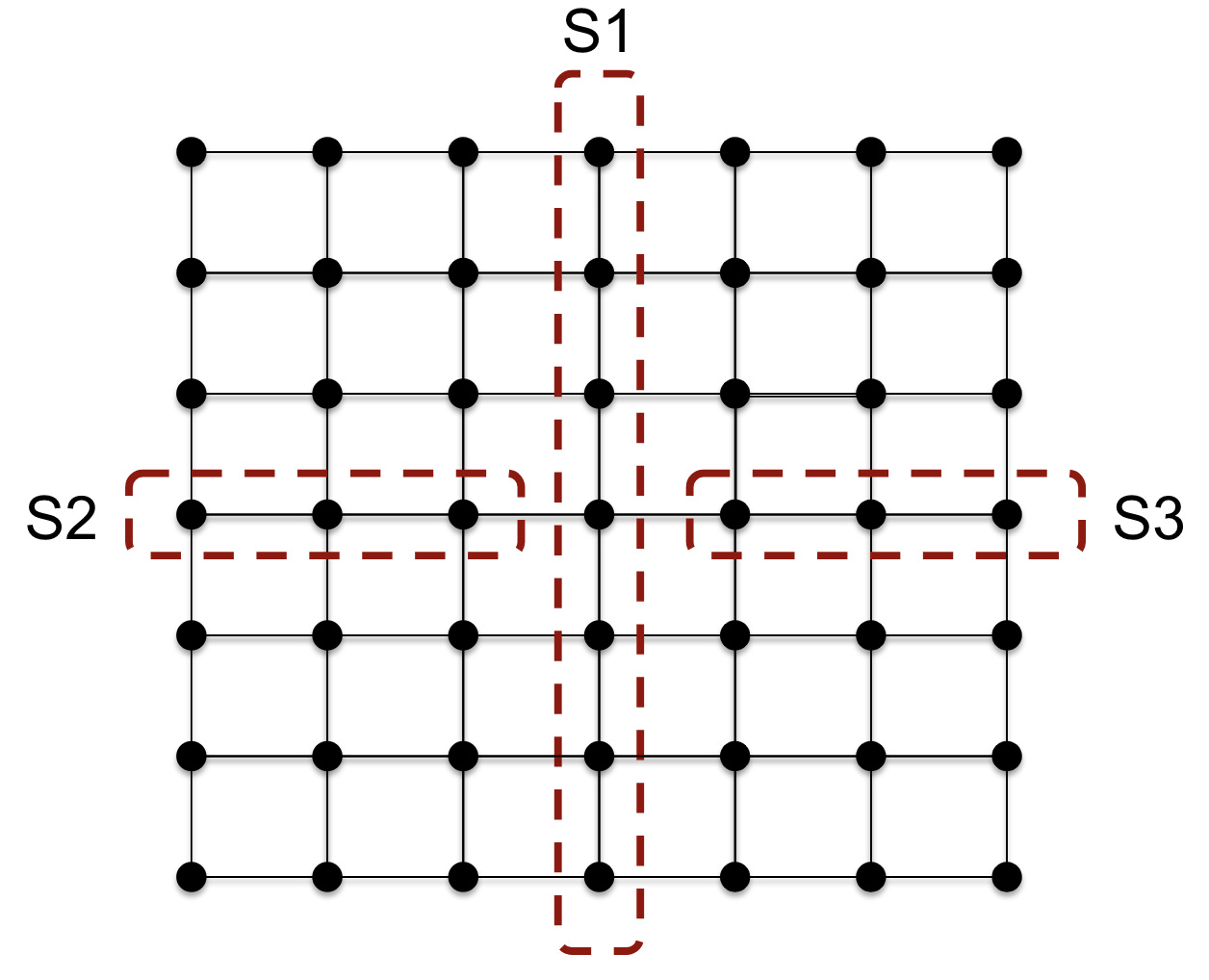}}
\hspace{1cm}
\scalebox{0.16}{\includegraphics{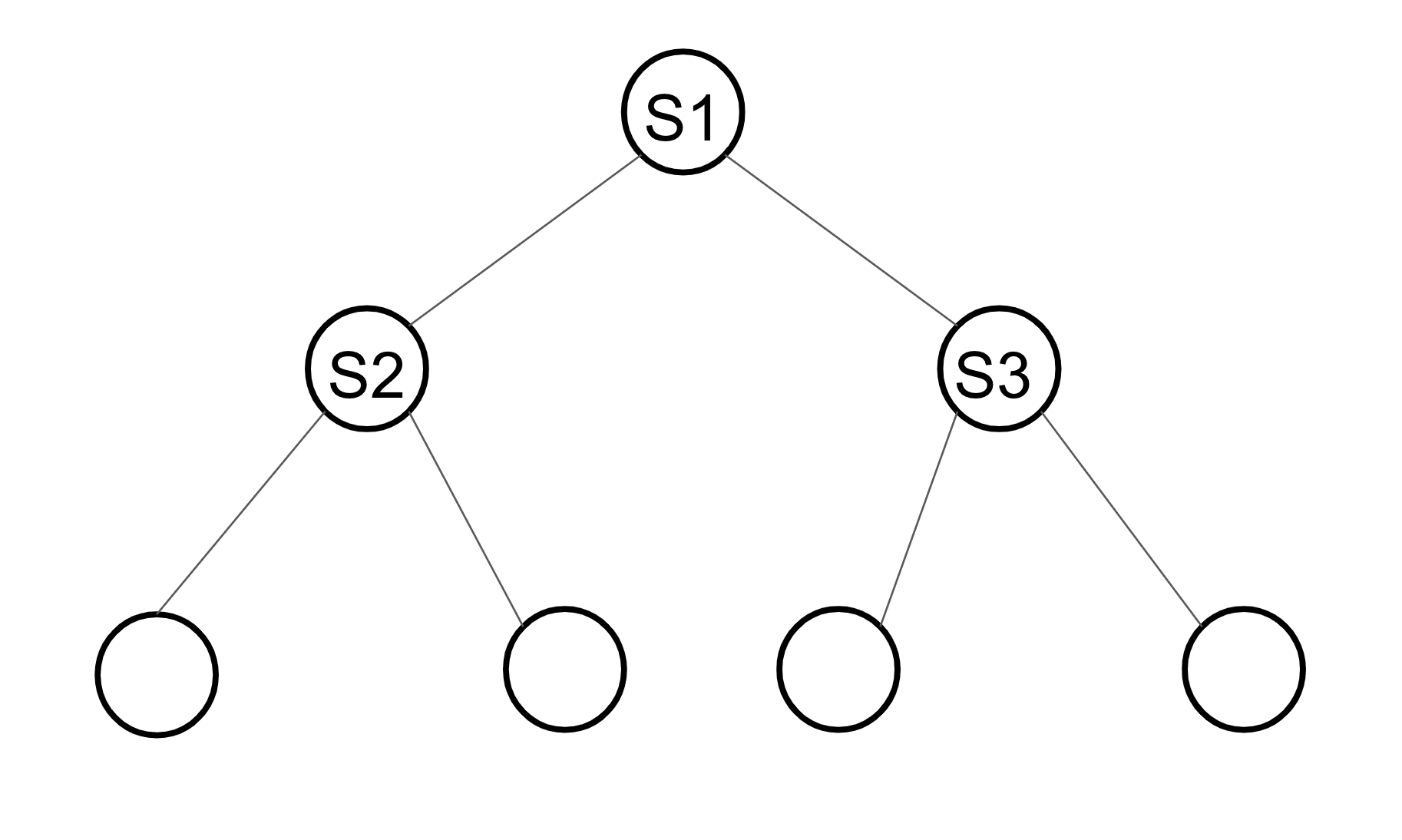}
}
\caption{(Left) an example of the graph of $A$ and its nested dissection partitioning. S1 is the top {separator}, S2 and S3 are two decoupled {separators} at the second level, and the remaining four parts are decoupled from each other. (Right) nested-dissection tree and task graph. {Task dependency}: every node {depends} on its children (if exist) and some descendants, and nodes at the same level can execute {in parallel}.}
\label{f:nd}
\end{center}
\end{figure}

\begin{algorithm}
\caption{Compute ordering}
\label{a:order}
\begin{algorithmic}[1]
\Require irreducible Laplacian matrix $L \in \mathbb{R}^{N \times N}$ and number of threads $p$
\Ensure the nested-dissection tree $\mathcal{T}$
\Statex
\State $\ell = \log_2(p)$
\hfill \Comment{\cyan{// assume $p$ is a power of 2}}
\State Create a full binary tree $\mathcal{T}$ of $\ell$ levels
\hfill \Comment{\cyan{// initialize output}}
\State \Call{ComputeOrdering}{$\mathcal{T}\rightarrow$root, $L$, $\ell$}
\hfill \Comment{\cyan{// start recursion}}
\Statex
\Function{ComputeOrdering}{node, $L$, $\ell$}
    \If{$\ell > 0$}
    \Statex \Comment{\quad\quad\quad \cyan{// partition graph/indices into ``left'', ``right'', and ``seperator''}}
    \State  $\mathcal{I}_l, \, \mathcal{I}_r, \, \mathcal{I}_s$ = \Call{PartitionGraph}{$L$}
    \hfill \Comment{\cyan{// call METIS}}
    \State node$\rightarrow$store\_indices($\mathcal{I}_s$)
    \State \Call{ComputeOrdering}{node$\rightarrow$left, $L(\mathcal{I}_l, \mathcal{I}_l)$, $\ell-1$}
    \State \Call{ComputeOrdering}{node$\rightarrow$right, $L(\mathcal{I}_r, \mathcal{I}_r)$, $\ell-1$}
    \Else
    \State $\mathcal{I}$ = \Call{ComputeAMD}{$L$}
    \hfill \Comment{\cyan{// AMD ordering at leaf level}}
    \State node$\rightarrow$store\_indices($\mathcal{I}$)
    \EndIf 
\EndFunction
\end{algorithmic}
\end{algorithm}

The nested dissection partitioning is naturally associated with a tree structure, where leaf nodes correspond to disconnected regions and the other nodes correspond to separators at different levels; see \cref{f:nd} (right). This tree maps to the task graph of a parallel algorithm: every tree node/task stands for applying \cref{a:rchol} to associated rows/columns in the sparse matrix. It is obvious that tasks at the same level can execute in parallel. Notice a task depends on not only  its children but also some of their descendants. We employ a multi-frontal type approach~\cite{liu1992multifrontal} in our parallel algorithm, where a task receives the Schur complement updates from  its two children and sends necessary updates to  its parent. In other words, a task communicates  with only its children and parent. The pseudocode is shown in \cref{a:parallel}, where we traverse the task tree in post order to generate all tasks.

We have implemented \cref{a:parallel} with both OpenMP\footnote{\url{https://www.openmp.org/}} tasks and the C++ thread library\footnote{\url{https://en.cppreference.com/w/cpp/thread}}, and we found the latter delivered slightly better performance in our numerical tests. Specifically, we use \texttt{std::async} to launch an asynchronous task at Line~\ref{l:task} on a new thread and store the results in an \texttt{std::future} object. Synchronization is achieved by calling the \texttt{get()} method on the previous \texttt{future} object at Line~\ref{l:sync}. One advantage of our approach is that we are able to pin threads on cores for locality via \texttt{sched\_setaffinity()} in \texttt{sched.h}.


\begin{algorithm} 
\caption{Parallel randomized Cholesky factorization}
\label{a:parallel}
\begin{algorithmic}[1]
\Require irreducible Laplacian matrix $L \in \mathbb{R}^{N \times N}$ and the nested-dissection tree $\mathcal{T}$
\Ensure matrix $G \in \mathbb{R}^{N \times N}$ (lower triangular if reordered according to $\mathcal{T}$)
\Statex
\State \Call{ParRchol}{$\mathcal{T}\rightarrow$root, $L$, $G$}
\hfill \Comment{\cyan{// start recursion; $L$ and $G$ modified in place}}
\Statex
\Function{ParRchol}{node, $L$, $G$}
\hfill \Comment{\cyan{// post-order tree traversal}}
    \If{node$\rightarrow$not\_leaf()}
    \Statex \Comment{\quad\quad\quad \cyan{// recursive task generation}}
    \State $S_l$ = \Call{ParRchol}{node$\rightarrow$left, $L$, $G$} \label{l:task}
    \State $S_r$ = \Call{ParRchol}{node$\rightarrow$right, $L$, $G$}
    \EndIf
    \State \Comment{\cyan{// wait until child tasks finish}} \label{l:sync}
    \State $L = L + S_l + S_r$ 
    \hfill \Comment{\cyan{// merge updates from children (reduction)}}
    \State $\mathcal{I}$ = node$\rightarrow$get\_indices()
    \State $S$ = \Call{RcholBlock}{$\mathcal{I}$, $L$, $G$}
    \hfill \Comment{\cyan{// apply \rchol{} to a block of indices}}
    \State \Return $S$
\EndFunction
\Statex
\Function{RcholBlock}{$\mathcal{I}$, $L$, $G$}
    \State $S = \bm{0}_{N \times N}$
    \For {$k \in \mathcal{I}$}
    \State \Comment \cyan{// $\ell_{kk} = 0$ at the last index in the top separator according to \cref{th:stable}}
    \State $ G(:,k) = \left\{ \begin{array}{llr} 
    		L(:,k)/\sqrt{\ell_{kk}} & \ell_{kk} \not = 0 \\
		\bm{0} & \ell_{kk} = 0
    		 \end{array} \right. $
    \State $C = \Call{SampleClique}{L, k}$
    \State $C_1, C_2 = \Call{SeparateEdges}{\mathcal{I}, C}$
    \hfill \Comment \cyan{// $C_1+C_2=C$}
    \State $L = L - L^{(k)} + C_1$
    \State $S = S + C_2$
    \hfill \Comment \cyan{// cumulate updates and send to parent}
    \EndFor
    \State \Return $S$
\EndFunction
\Statex
\Function{SeparateEdges}{$\mathcal{I}$, $C$}
    \State $C_1 = \bm{0}_{N \times N}, C_2 = \bm{0}_{N \times N}$
    \Statex \Comment \quad\, \cyan{// suppose $C = \sum_{e_{ij} \in E} w_{ij} \, \bb_{ij} \bb_{ij}^\top$ since $C$ is a graph Laplacian}
    \For {$e_{ij} \in E$}
    \If {$i \in \mathcal{I}$ \textbf{or} $j \in \mathcal{I}$}
    \State $C_1 = C_1 + w_{ij} \, \bb_{ij} \bb_{ij}^\top$
    \hfill \Comment \cyan{// needed by the current node}
    \Else
    \State $C_2 = C_2 + w_{ij} \, \bb_{ij} \bb_{ij}^\top$
    \hfill \Comment \cyan{// needed by ancestors}
    \EndIf
    \EndFor
    \State \Return $C_1, C_2$
\EndFunction
\end{algorithmic}
\end{algorithm}

%

\section{Numerical Results} \label{s:result}

In this section, we refer to our randomized preconditioner  as \rchol. Recall our goal is solving $Ax=b$, and our approach is constructing a preconditioner $GG^\top$, where $G$ is a lower triangular matrix.


Besides problems from the SuiteSparse Matrix Collection, we generate test matrices from discretizing Poisson's equation, variable-coefficient Poisson's equation, and anisotropic Poisson's equation:
\begin{equation} \label{e:poisson}
- \nabla \cdot ( a(x) \, \nabla u(x)) = f, \quad x \in \Omega=[0,1]^3, \quad u(x)=0 \text{ on } \partial \Omega.
\end{equation}
\begin{itemize}
\item
Poisson's equation:  $a(x) = 1$. 

\item
Variable-coefficient Poisson's (VC-Poisson) equation: we generate a high-contrast coefficient field $a(x)$ following~\cite{ho2016hierarchical,chen2018distributed,cambier2020algebraic}.  First, we generate $\{a_i\}$ from standard uniform distribution on a regular grid and compute the median $\mu$. Then, we convolve $\{a_i\}$ with an isotropic Gaussian of width $4h$, where $h$ is the grid spacing. Last, we quantize $\{a_i\}$ by setting
\begin{equation} \label{e:a}
a_i = \left\{
\begin{array}{cc}
 \rho^{1/2}, & \text{if } a_i \ge \mu, \\
 \rho^{-1/2}, & \text{if } a_i < \mu. \\
\end{array}
\right.
\end{equation}
See \cref{s:e} for an example of the random coefficients.

\item
anisotropic Poisson's (Aniso-Poisson) equation: {$a(x) = \diag(\delta^{1/2}, 1, \delta^{-1/2})$, where the coefficients are constant along each dimension.}

\end{itemize}
{In particular, we discretize the above elliptic PDE using standard 7-point finite difference stencil over a uniform $n \times n \times n$ grid. Let $h=1/n$, $x_j = h(j_1, j_2, j_3)$, where $j$ is the index of the triplet $(j_1, j_2, j_3)$ for $1\le j_1, j_2, j_3 \le n$. The discretized PDE reads:
\begin{align*}
(a_{j-e_1/2} + a_{j+e_1/2} 
+ a_{j-e_2/2} + a_{j+e_2/2} 
+ a_{j-e_3/2} + a_{j+e_3/2} ) u_j \\
- a_{j-e_1/2} u_{j-e_1} + a_{j+e_1/2} u_{j+e_1} 
- a_{j-e_2/2} u_{j-e_2} + a_{j+e_2/2} u_{j+e_2} \\
- a_{j-e_3/2} u_{j-e_3} + a_{j+e_3/2} u_{j+e_3} 
 & = h^2 f_j,
\end{align*}
where $e_1=(1,0,0), e_2=(0,1,0), e_3=(0,0,1)$, and $u_j \approx u(x_j)$ is to be solved.}

Experiments were performed on a node from Frontera\footnote{\url{https://frontera-portal.tacc.utexas.edu/user-guide/}}. Results  in \cref{s:r1} and \cref{s:r2} were obtained using a single thread on an Intel Xeon Platinum 8280 (``Cascade Lake"), and results in \cref{s:r3}  were obtained using multiple threads/cores on  an Intel Xeon Platinum 8280M.
Below are the notations we use to report results (all timing results are in seconds):

\begin{itemize}
    \item $N$: matrix size of $A$.
    \item $p$: number of threads/cores.
    \item nnz: number of non-zeros in $A$.
    \item fill: \emph{twice} the number of non-zeros in $G$.
    \item $t_p$: time for computing a permutation/reordering for $A$.
    \item $t_f$: time for computing the factorization/preconditioner.
    \item $t_s$: \emph{total PCG time}  for solving a standard-uniform random $b$.
    \item $n_{it}$: number of the PCG iterations with tolerance \num{1e-10}. In cases where PCG stagnated before convergence, we report the iteration number to stagnation and the corresponding relative residual (relres) $\|b-Ax\|_2/\|b\|_2$.
\end{itemize}

\subsection{Reordering and Stability} \label{s:r1}

We present results for five commonly-used reordering strategies in \cref{t:order}. The test problem is the standard 7-point finite-difference discretization of Poisson's equation in a unit cube with the Dirichlet boundary condition. {We have also tested the five strategies on other problems including VC-Poisson, Aniso-Poisson, and problems from SuiteSparse Matrix Collection (see \Cref{s:r_ufl}), and the following observations generally apply.}
\begin{enumerate}
\item
natural ordering (a.k.a., lexicographic ordering)/no reordering leads to significant amount of fill-in. Although PCG required a small number of iterations, the total solve-time is significant with a relatively dense preconditioner.
\item
reverse Cuthill-McKee ordering aims at a small bandwidth of the reordered matrix, which helps reduce fill-in for some applications. But results showed that it is was not effective for \rchol{}.
\item
random ordering as suggested in~\cite{kyng2016approximate} is effective in fill-in reduction. However, it results in widely scattered sparsity pattern in the triangular factor as shown in \cref{f:sparsity}, hampering practical performance of triangular solves at every iteration.
\item
nested dissection (ND) ordering is effective in fill-in reduction but requires significant time to compute.
\item
approximate minimum degree (AMD) ordering~\cite{amestoy1996approximate} is also effective in fill-in reduction and can be computed quickly. {The fill-in pattern of \rchol{} is not deterministic and is different from the (exact) Cholesky factorization.  Although the AMD is designed as a greedy strategy for minimizing the fill-in of the (exact) Cholesky factorization, it also performs well when used with \rchol{}. Among the five reordering strategies considered here, the AMD leads to the minimum running time consistently for all of our test problems, so we use the AMD by default.}
\end{enumerate}


Although \rchol{} uses randomness in the algorithm, the resulting preconditioner delivers extremely consistent performance as \cref{t:trial} shows.

\begin{table}
    \centering 
    \caption{\em Sparse matrix reordering. The matrix is from discretizing Poisson's equation on a 3D regular grid of size $256^3$ using standard 7-point finite difference. The orderings are computed using Matlab commands in the parentheses.}
    \label{t:order}
    \begin{tabular}{cccccc} 
    \toprule
    Ordering & fill/nnz & $t_p$ & $t_f$ & $t_s$ & $n_{it}$ \\
    \midrule 
    no reordering & 10.2 & 0 & 139 & 173 & 39 \\
    reverse Cuthill-McKee (symrcm) & 7.9 & 5 & 97 & 138 & 41 \\
    random ordering (randperm) & 3.3 & 0.8 & 76 & 362 & 55 \\
    nested dissection (dissect) & 3.3 & 206 & 66 &  132 & 65   \\
    \rowcolor{Gray}
    approximate minimum degree (amd) & 3.5 & 38 & 50 & 126 & 60   \\
    \bottomrule
    \end{tabular}
\end{table}

\begin{figure}
\begin{subfigure}{.45\textwidth}
  \centering
  \includegraphics[width=.4\linewidth]{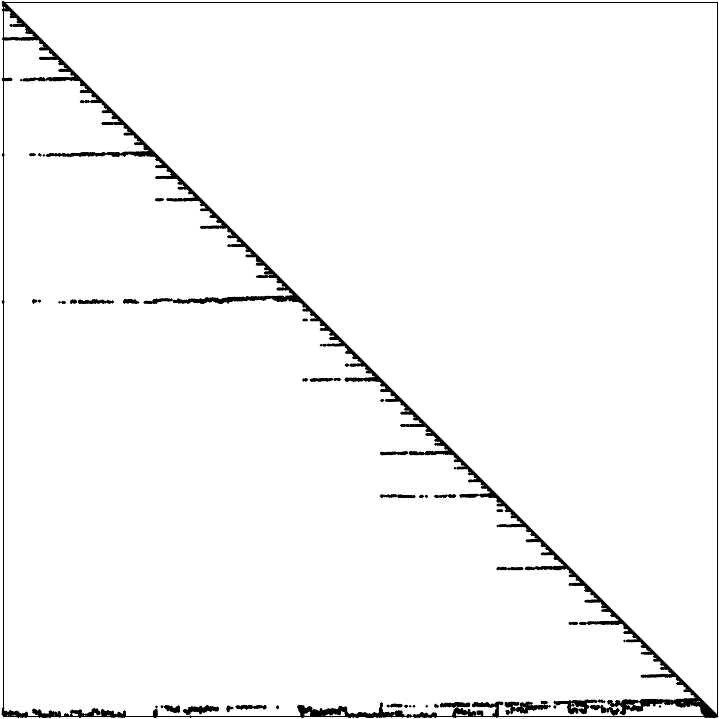}  
  \caption{AMD reordering: \num{2.1e+8} non-zeros} 
  \label{fig:sub-first}
\end{subfigure}
\begin{subfigure}{.45\textwidth}
  \centering
  \includegraphics[width=.4\linewidth]{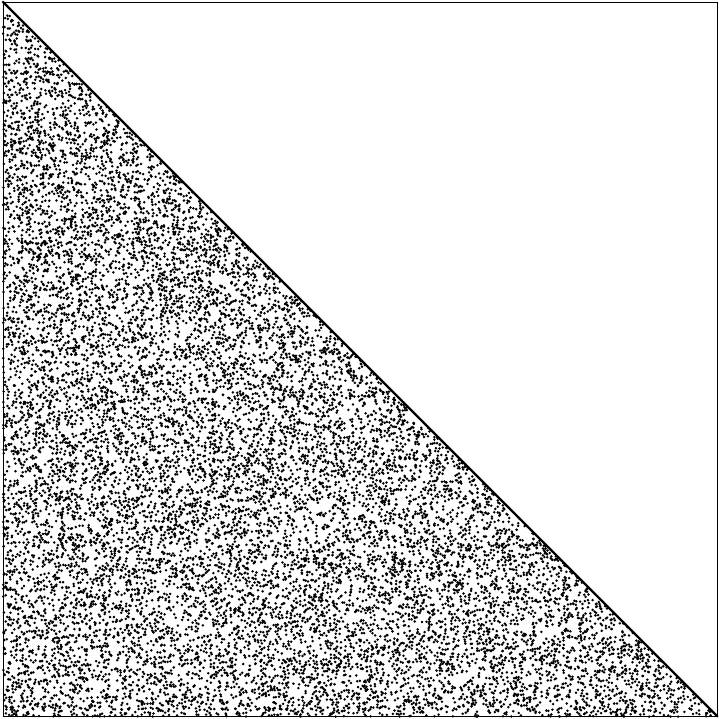}  
  \caption{random reordering: 1.9e+8 non-zeros} 
  \label{f:nnz_random}
\end{subfigure}
\caption{Sparsity pattern of triangular factors computed by \rchol{} corresponding to the AMD ordering and the random ordering in \cref{t:order}, respectively. (The full spy plot for random ordering is quite large, and (b) corresponds to the leading principle submatrix of size 3e+5.)}
\label{f:sparsity}
\end{figure}

    
\begin{table}
    \centering 
    \caption{\em Variance of \rchol{} (minimums and maximums among 10 independent trials). The matrices are from discretizing Poisson, VC-Poisson ($\rho=\num{1e+05}$) and Aniso-Poisson ($\delta=\num{1e+04}$) on a 3D regular grid of size $256^3$ using standard 7-point finite difference. (PCG tolerance is \num{1e-06} for VC-Poisson, a highly ill-conditioned problem; see \Cref{s:vc_poisson}.)}
    \label{t:trial}
    \begin{tabular}{cgccg} 
    \toprule
    Ordering & fill/nnz & $t_f$ & $t_s$ & $n_{it}$  \\
    \midrule 
    Poisson  
	& 3.538 - 3.542 & 48 - 54 & 117 - 128 & 57 - 62 \\
    VC-Poisson 
        & 4.074 - 4.078 & 56 - 65 & 257 - 303 & 120 - 141 \\
    Aniso-Poisson     
        & 2.556 - 2.557 & 38 - 43 & 79 - 80 & 44 - 44 \\
    \bottomrule
    \end{tabular}
\end{table}

\subsection{Comparison with incomplete Cholesky} \label{s:r2}

We compare \rchol{} to the incomplete Cholesky preconditioner with thresholding dropping (\ichol{}) in MATLAB\textsuperscript{\textregistered} R2020a. 
In particular, we manually tuned the drop tolerance in \ichol{} to obtain  preconditioners with slightly more fill-in. For both preconditioners, the construction time is usually much smaller than the time spent in PCG. For every PCG iteration, we expect similar running time because both preconditioners have approximately the same amount of fill-in. Therefore, the performance depends mostly on the numbers of PCG iterations. We used the AMD ordering in \rchol{}. Based on our experiments, \ichol{} performed better without any reordering, which is consistent with  empirical results observed in the literature~\cite{duff1989effect}.

\subsubsection{Matrices from SuiteSparse Matrix Collection} \label{s:r_ufl}

We first compare \rchol{} with \ichol{} on four SPD matrices from SuiteSparse Matrix Collection\footnote{\url{https://sparse.tamu.edu/}} that are not necessarily SDD. {The first is an SDDM matrix, the second a SDD matrix, and the last two SPD (but not SDD) matrices. All matrices have only negative off-diagonal entries except for the second matrix.} The second matrix is SDD but approximately a third of the off-diagonal entries are as small as \num{3.2e-7}. Since these entries are quite small relative to the remaining entries, we simply ignored these positives when applying \rchol{}. The last two matrices are not SDD, and some of the diagonals are smaller than the sum of the absolute value of off-diagonals. But we were able to run \rchol{} in a ``black-box'' fashion, which is equivalent to adding diagonal compensations to make the original matrix SDD.


Without any preconditioner, CG converged extremely slowly as shown in \cref{t:ufl1}. 
As \cref{t:ufl2} shows, although the highly-optimized \ichol{} (in MATLAB) delivers faster factorization than our implementation of \rchol{}, the \rchol-PCG took much less time than the \ichol-PCG due to significantly less iterations. In particular, PCG took about $9\times$ more iterations with \ichol{} for ``ecology2''. For all cases with \ichol{}, PCG stagnated before the $\num{1e-10}$ tolerance was reached. With \rchol{}, the relative residuals decreased to below $\num{1e-10}$ for the second and the last problems. {We also tested \ichol{} with no fill-in, and the total time were greater than those in \cref{t:ufl2}.}

\begin{table}
    \centering 
    \caption{\em SPD matrices from SuiteSparse Matrix Collection. {With \emph{no preconditioner}, CG converged extremely slow, and the relative residuals were still quite large after 2500 iterations except for the second problem.}}
    \label{t:ufl1}
    \begin{tabular}{cccccgc} 
    \toprule
    & Name & $N$ & nnz & Property & $n_{it}$ & relres \\
    \midrule 
    \# 1 & ecology2 & 1.0e+5 & 5.0e+6 & SDDM & 2500 & {1e-01} \\
    \# 2 & parabolic\_fem & 5.3e+5 & 3.7e+6 & SDD & 2500 & {2e-07} \\
    \# 3 & apache2 & 7.2e+5 & 4.8e+6 & not SDD & 2500 & {1e-02} \\
    \# 4 & G3\_circuit & 1.6e+6 & 7.7e+6 & not SDD & 2500 & {5e-01} \\
    \bottomrule
    \end{tabular}
    \caption{\em Comparison between \rchol{} preconditioner and \ichol{} preconditioner on matrices from SuiteSparse Matrix Collection. AMD ordering is applied with \rchol{}. Based on our experiments, the vanilla \ichol{} preconditioner without any reordering performs slightly better than with a reordering.}
    \label{t:ufl2}
    \begin{tabular}{l|ccccgc|cccgc} 
    \toprule
    \multirow{2}{*}{}
    & \multicolumn{6}{c|}{\rchol{}} & \multicolumn{5}{c}{\ichol{}} \\
    & fill/nnz & $t_p$ & $t_f$ & $t_s$ & $n_{it}$ & relres & fill/nnz & $t_f$ & $t_s$ & $n_{it}$  & relres \\ 
    \midrule 
    \# 1 & 2.41 & 0.4 & 1.4 & 6.3 & 89 & 1e-08 & 2.72 & 0.2 & 68 & 798 & 3e-08 \\
    \# 2 & 2.27 & 0.4 & 1.0 & 2.8 & 65 & 8e-11 & 2.29 & 0.2 & 15 & 411 & 2e-10 \\
    \# 3 & 2.93 & 0.6 & 1.5 & 4.1 & 63 & 3e-10 & 2.96 & 0.2 & 18 & 322 & 4e-10 \\
    \# 4 & 2.68 & 1.5 & 2.8 & 9.6 & 90 & 9e-11 & 2.75 & 0.3 & 40 & 379 & 2e-10 \\
    \bottomrule
    \end{tabular}
\end{table}

\subsubsection{Variable-coefficient Poisson's equation} \label{s:vc_poisson}

%

{We compare the \rchol{} preconditioner with the \ichol{} preconditioner on a sequence of SDDM matrices that become gradually more ill-conditioned.}
The discretization of VC-Poisson on a regular grid using the standard 7-point finite-difference stencil has a condition number $\mathcal{O}(\rho N^{2/3})$. 

The results are similar to above, where \ichol{} required at least twice as many iterations. As a result, the total time taken with the \rchol{} preconditioner is much less than with the \ichol{} preconditioner in all cases. In \cref{t:poisson}, when the condition number is large, PCG stoped  progressing  before reaching the tolerance \num{1e-10}. Consequently, the relative residual with the solution returned from PCG decreased from approximately \num{1e-11} to approximately \num{1e-8} as $\rho$ increases from \num{1e+0} to \num{1e+5}. Both  preconditioners suffer from this performance deterioration.

\begin{table}
    \centering 
    \caption{\em Comparison between \rchol{} preconditioner and \ichol{} preconditioner on matrices from discretizing variable-coefficient Poisson's equation on a regular grid of size $128^3$ using standard 7-point finite difference ($N$=\num{2.0e+06}, nnz=\num{1.4e+07}). The coefficients have contrast ratio $\rho$; see \cref{e:a}. 
    When $\rho \ge \num{1e+3}$, PCG stagnated before reaching tolerance $\num{1e-10}$.}
    \label{t:poisson}
    \begin{tabular}{c|ccccgH|cccgH} 
    \toprule
    \multirow{2}{*}{$\rho$}
    & \multicolumn{6}{c|}{\rchol{}} & \multicolumn{5}{c}{\ichol{}} \\
    & fill/nnz & $t_p$ & $t_f$ & $t_s$ & $n_{it}$ & relres & fill/nnz & $t_f$ & $t_s$ & $n_{it}$ & relres \\ 
    \midrule
    1e+0 & 3.23 & 3.8 & 5.3 & 12 & 51   & 9e-11 & 3.40 & 0.7 & 21 & 102 & 9e-11 \\
    1e+1 & 3.42 & 3.8 & 5.6 & 13 & 53   & 7e-11 & 3.46 & 0.8 & 37 & 175 & 9e-11  \\
    1e+2 & 3.57 & 3.8 & 5.7 & 19 & 83   & 9e-11 & 3.63 & 0.8 & 50 & 235 & 7e-11  \\
    1e+3 & 3.62 & 3.8 & 5.7 & 28 & $115$ & 2e-10 & 3.72 & 0.9 & 57 & $260$ & 2e-10  \\ 
    1e+4 & 3.62 & 3.9 & 5.7 & 29 & $126$ & 5e-09 & 3.78 & 0.9 & 57 & $254$ & 5e-09  \\ 
    1e+5 & 3.62 & 3.9 & 5.8 & 32 & $144$ & 5e-08 & 3.78 & 0.9 & 63 & $272$ & 7e-08 \\ 
    \bottomrule
    \end{tabular}
\end{table}


\subsection{{Comparison to multigrid methods}} \label{s:aniso_poisson}

We compared \rchol{} to three multigrid methods including the combinatorial multigrid (CMG)~\cite{koutis2011combinatorial}\footnote{\url{http://www.cs.cmu.edu/~jkoutis/cmg.html}}, the Ruge-Stuben (classical) AMG (RS-AMG) and the smoothed aggregation AMG (SA-AMG). The RS-AMG and the SA-AMG are from the pyamg package~\cite{OlSc2018}\footnote{\url{https://github.com/pyamg/pyamg}}. 
We ran \rchol{} through the C++ interface. 

The test matrices include the four problems from the SuiteSparse Matrix Collection (see \cref{s:r_ufl}) and three matrices of size $128^3$ from discretizing the three Poisson equations, respectively. The results of comparison are shown in \cref{t:multigrid}, which shows that our method is the fastest for two of the problems, CMG is the fastest for one problem, and the classical AMG is the fastest for the other four problems. 

{As is well accepted by the scientific computing community, the performance of linear solvers may depend on the input matrices, and there is no single best solver for all problems. As a result, there exists different solvers/preconditioners including incomplete factorizations, multigrid, sparse direct solvers, etc. As \cref{t:multigrid} shows, multigrid methods  usually perform well on matrices corresponding to regular grids.}



\begin{table}
    \centering 
    \caption{Comparison to multigrid methods. Highlighted rows are the fastest solve time among all methods. See \Cref{s:r_ufl} for the first four matrices. The remaining three matrices are discretizations of Poisson, VC-Poisson ($\rho=\num{1e+05}$), and Aniso-Poisson ($\delta=\num{1e+04}$) equations on a 3D regular grid of size $128 \times 128 \times 128$. }
    \label{t:multigrid}
    \begin{minipage}{.5\linewidth}
    \begin{tabular}{cHHcccH} 
    \toprule
      \multirow{2}{*}{matrix}
    & \multicolumn{6}{c}{\rchol{}}  \\
     & fill/nnz & $t_p$ & $t_f$ & $t_s$ & $n_{it}$ & $res$ \\ 
    \midrule 
    ecology2 & 2.41 & 0.4 & 0.9 & 4.63 & 90 & 1e-08  \\
    parabolic\_fem & 2.27 & 0.4 & 0.9 & 2.08 & 67 & 1e-10  \\
    \rowcolor{Gray}
    {apache2} & 2.93 & 0.6 & 1.4 & 2.91 & 64 & 5e-10  \\
    \rowcolor{Gray}
    {G3\_circuit} & 2.68 & 1.4 & 2.5 & 7.96 & 90 & 1e-10  \\
    Poisson & 2.68 & 1.4 & 6.1 & 8.07 & 53 & 7e-11  \\
    VC-Poisson & 2.68 & 1.4 & 6.6 & 20.7 & 131 & 5e-08  \\
    Aniso-Poisson & & & 3.71 & 4.88 & 36 & 5.e-11 \\
    \bottomrule
    \end{tabular}
    \end{minipage}%
    \begin{minipage}{.5\linewidth}
    \begin{tabular}{cHHcccH} 
    \toprule
     \multirow{2}{*}{matrix}
    & \multicolumn{6}{c}{\texttt{CMG}}  \\
     & fill/nnz & $t_p$ & $t_f$ & $t_s$ & $n_{it}$ & $res$ \\ 
    \midrule 
    ecology2 & 2.41 & 0.4 & 1.0 & 4.27 & 58 & 9e-09  \\
    parabolic\_fem & 2.27 & 0.4 & 2.59 & 3.20 & 45 & 4e-10  \\
    apache2 & 2.93 & 0.6 & - & - & - & - \\
    G3\_circuit & 2.68 & 1.4 & 5.67 & 9.59 & 73 & 9e-11  \\
    Poisson & 2.68 & 1.4 & 7.51 & 7.60 & 43 & 7e-11  \\
    \rowcolor{Gray}
    {VC-Poisson} & 2.68 & 1.4 & 9.25 & 10.88 & 62 & 3e-08  \\
    Aniso-Poisson & & & 6.20 & 8.90 & 67 & 1e-10 \\
    \bottomrule
    \end{tabular}
    \end{minipage}

    \noindent
    \begin{minipage}{.5\linewidth}
    \begin{tabular}{cHHcccH} 
    \toprule
     \multirow{2}{*}{matrix}
    & \multicolumn{6}{c}{RS-AMG}  \\
     & fill/nnz & $t_p$ & $t_f$ & $t_s$ & $n_{it}$ & $res$ \\ 
    \midrule 
    \rowcolor{Gray}
    {ecology2} & 2.41 & 0.4 & 1.44 & 3.00 & 21 & 3e-09  \\
    \rowcolor{Gray}
    {parabolic\_fem} & 2.27 & 0.4 & 1.08 & 1.15 & 14 & 2e-11  \\
    apache2 & 2.93 & 0.6 & 1.17 & 13.38 & 101 & 5e-10 \\
    G3\_circuit & 2.68 & 1.4 & 2.38 & 10.82 & 39 & 5e-11  \\
    \rowcolor{Gray}
    {Poisson} & 2.68 & 1.4 & 6.15 & 5.34 & 13 & 9e-12  \\
    VC-Poisson & 2.68 & 1.4 & 6.55 & 15.68 & 38 & 1e-08  \\
    \rowcolor{Gray}
    {Aniso-Poisson} & & & 3.34 & 4.09 & 9 & 2e-12 \\
    \bottomrule
    \end{tabular}
    \end{minipage}%
    \begin{minipage}{.5\linewidth}
    \begin{tabular}{cHHcccH} 
    \toprule
     \multirow{2}{*}{matrix}
    & \multicolumn{6}{c}{SA-AMG}  \\
     & fill/nnz & $t_p$ & $t_f$ & $t_s$ & $n_{it}$ & $res$ \\ 
    \midrule 
    ecology2 & 2.41 & 0.4 & 3.30 & 2.54 & 19 & 2e-09  \\
    parabolic\_fem & 2.27 & 0.4 & 1.42 & 1.48 & 27 & 1e-10  \\
    apache2 & 2.93 & 0.6 & 2.91 & 6.68 & 49 & 4e-10 \\
    G3\_circuit & 2.68 & 1.4 & 7.29 & 23.43 & 67 & 1e-10  \\
    Poisson & 2.68 & 1.4 & 10.20 & 7.80 & 17 & 1e-11  \\
    VC-Poisson & 2.68 & 1.4 & 9.74 & 14.20 & 32 & 1e-08  \\
    Aniso-Poisson & & & 9.46 & 44.14 & 101 & 3e-07 \\
    \bottomrule
    \end{tabular}
    \end{minipage}    
\end{table}

\subsection{Parallel scalability} \label{s:r3}

In this section, we show the speedup of running \rchol{} with multiple threads and the stability of the resulting preconditioner in terms of the fill-in ratio and the PCG iteration. The test problem is solving the 3D Poisson's equation with the Dirichlet boundary condition in the unit cube, which is discretized using the 7-point stencil on regular grids. {We ran \rchol{} in single-precision floating-point arithmetic to reduce memory footprint and computation time, and we ran PCG in double precision. The use of single precision in the construction of preconditioners has been studied in the literature~\cite{giraud2008mixed,lindquist2020improving,abdelfattah2020survey,loe2021experimental}, which may lead to an increase of PCG iterations for difficult problems. Here, our results show that the use of single precision in \rchol{} does not impact the number of PCG iterations for solving discretized Poisson's equation.}

With $p=1$ thread, we used the AMD reordering; otherwise when $p>1$, we used a $\log_2 p$ -level ND ordering combined with the AMD ordering at the leaf level. All experiments were performed on an Intel Xeon Platinum 8280M (``Cascade Lake"), which has 112 cores on four sockets (28 cores/socket), and every thread is bound to a different core in a scattered fashion (e.g., the first four threads are each bound to one of the four sockets). We used the scalable memory allocator in the Intel TBB library.\footnote{\url{https://software.intel.com/content/www/us/en/develop/documentation/tbb-documentation/top/intel-threading-building-blocks-developer-guide/package-contents/scalable-memory-allocator.html}}

\cref{t:parallel} shows the results of three increasing problem sizes---the largest one being \emph{one billion} unknowns, and the factorization time scaled up to 64 threads in each case. (Results of parallel sparse triangular solves are given in \cref{s:d}.) For $N=1024^3$, the sequential factorization took nearly 42 minutes while it took approximately 3 minutes using 64 threads (cores), a $13.7\times$ speedup. \cref{t:parallel} also shows that the fill-in ratio and the PCG iteration are extremely stable regardless of the number of   threads  used. For the three problems, the memory footprint of the  preconditioners are about 1.7 GB, 15 GB and 130 GB, respectively,  in single precision, where we stored only a triangular  factor for every symmetric preconditioner.

\cref{f:parallel} shows the time spent on leaf tasks and separator tasks in  strong- and weak-scaling experiments, respectively; recall the task graph in \cref{f:nd}.
When $p$ doubles in strong scaling, the task tree increases by one level; in other words, every leaf task is  decomposed into two smaller leaf tasks plus a separator task. In addition, this decomposition computed algebraically by graph partitioning can hardly avoid load imbalance.  Therefore, the time reduction shrinks as $p$ increases in strong scaling.
When $p$ increases by $8\times$ in weak scaling, the task tree increases by three levels while the problem size associated with every leaf task remains the same if the partitioning is ideally uniform. In reality, however, load imbalance among leaf tasks becomes more and more significant as $p$ increases. The other reason for the increasing maximum running time of leaf tasks is that these tasks are memory-bound and suffer from memory bandwidth saturation if $p$ is large. The other bottleneck in weak scaling comes from the three extra levels of separator tasks when $p$ increases by $8\times$. Indeed, the top separator has size $\bigO(N^{2/3})$, but the corresponding task runs in sequential in our parallel algorithm. Parallelizing such tasks for separators at top levels is left as future work.

\cref{t:nitr} shows the effectiveness of the \rchol{} preconditioner computed with multiple threads, where the PCG iteration increases logarithmically with respect to the problem size $N$. By contrast, the PCG iteration with the \ichol{} preconditioner increases by approximately $2\times$ when the problem size $N$ increases by $8\times$ (the mesh is refined by $2\times$ in every dimension).


\begin{table}
    \centering 
    \caption{\em Parallel scalability on an Intel Cascade Lake that has 112 cores on four sockets. We applied \rchol{} to solving the 3D Poisson's equation (discretized with the 7-point stencil on regular grids). We used single-precision floating-point arithmetic in \rchol{}. 
    }
    \label{t:parallel}
    \begin{tabular}{c|ccc|ccc|ccc} 
    \toprule
    \multirow{2}{*}{$p$}
    & \multicolumn{3}{c|}{$N=256^3$} & \multicolumn{3}{c|}{$N=512^3$}  & \multicolumn{3}{c}{$N={1024^3}$} \\
    & fill/nnz & $t_f$ & $n_{it}$ & fill/nnz & $t_f$ & $n_{it}$ & fill/nnz & $t_f$ & $n_{it}$  \\ 
    \midrule
    1 & 3.56 & 19.9 & 57 & 3.93 & 226 & 65 & 4.31 & 2523 & 78 \\
    2 & 3.60 & 10.7 & 59 & 3.98 & 113 & 68 & 4.37 & 1279 & 79 \\
    4 & 3.61 & 5.7 & 57 & 3.98 & 58 & 65 & 4.39 & 664 & 75 \\
    8 & 3.63 & 3.3 & 61 &  3.99 & 35 & 65 &  4.38 & 388 & 75 \\
    16 & 3.66 & 2.3 & 59 & 4.00 & 23 & 65 & 4.38 & 258 & 76 \\
    32 & 3.66 & 1.9 & 57 & 4.02 & 18 & 64 & 4.39 & 197 & 71 \\
    64 & 3.66 & 1.7 & 57 & 4.02 & 16 & 67 & 4.38 & 184 & 75 \\
    \bottomrule
    \end{tabular}
\end{table}

\begin{figure}
\begin{center}
\scalebox{0.3}{\includegraphics{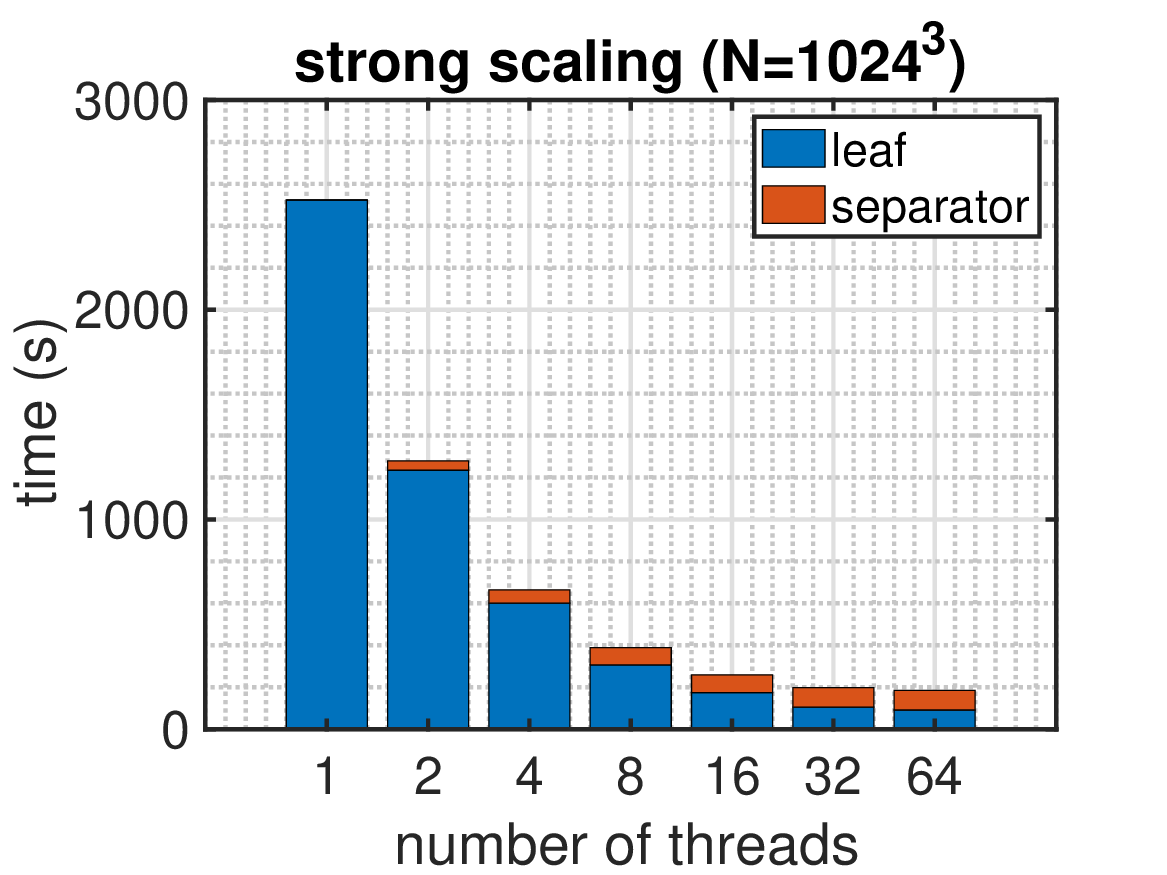}} 
\scalebox{0.3}{\includegraphics{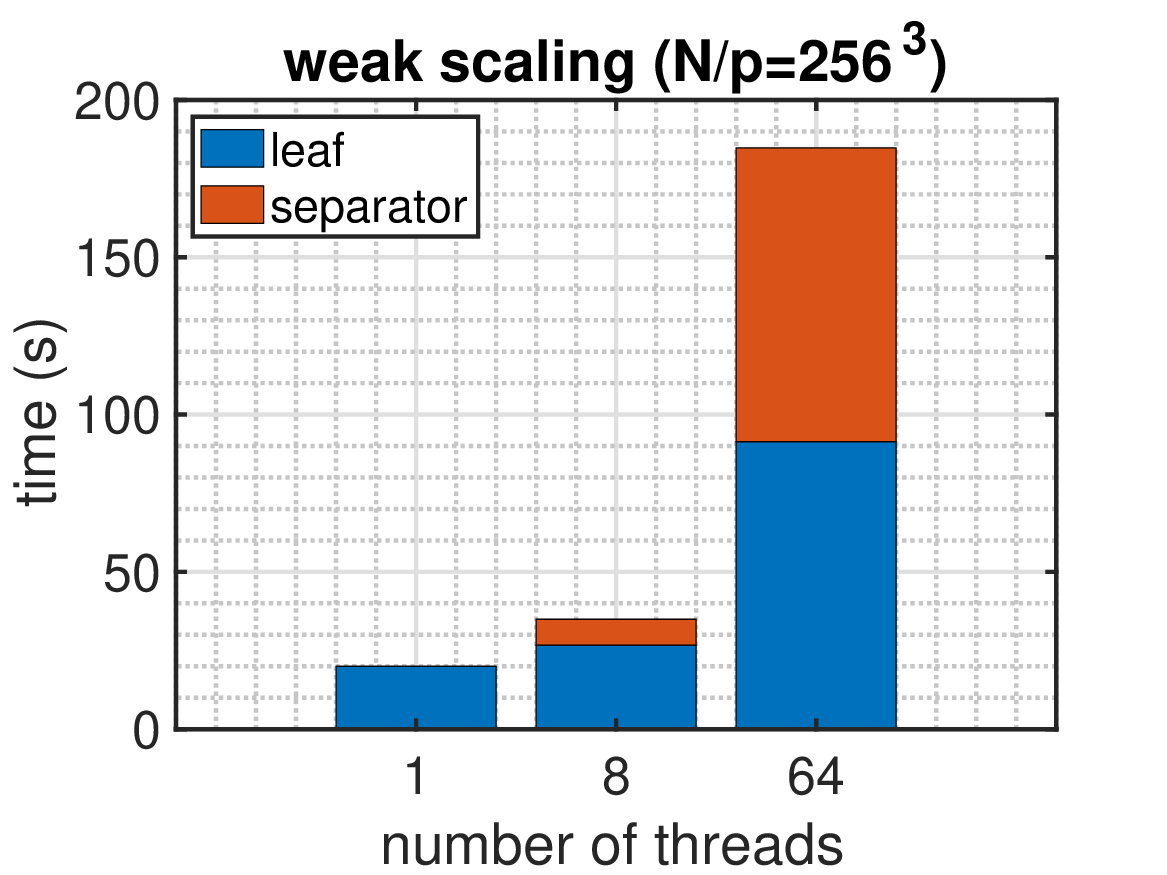}} 
\captionof{figure}{Strong and weak scalability of the \rchol{} factorization/construction time on an Intel Cascade Lake.
The input matrices are discretization of the 3D Poisson's equation using the 7-point stencil on regular grids.
We used single-precision floating-point arithmetic in \rchol{}.
``leaf'' denotes the maximum time of all leaf tasks executing in parallel, and ``separator'' denotes the remaining time spent on all separators. (Recall the task graph in \cref{f:nd}.)}
\label{f:parallel}
\end{center}
\end{figure}

\begin{table}
    \centering 
    \caption{\em Comparison of PCG iterations for solving the 3D Poisson equation discretized with the 7-point stencil on regular grids. 
    We did not run \ichol{} for $N=1024^3$ limited by our computation budget. 
    (We manually tuned the drop tolerance in \ichol{} to obtain preconditioners with slightly more fill-in. See \cref{t:parallel} for the fill-in of \rchol{} preconditioners.)}
    \label{t:nitr}
    \begin{tabular}{ccccc} 
    \toprule
     $N$ & $128^3$ & $256^3$ & $512^3$ & ${1024^3}$  \\
    \midrule
    \ichol{} & 100 & 185 & 341 & - \\    
    \rchol{} & 50 & 57 & 67 & 75 \\
    \bottomrule
    \end{tabular}
\end{table}

\section{Conclusions and generalizations} \label{s:end}

In this paper, we have introduced a preconditioner named \rchol{}  for solving SDD linear systems. To that end, we construct a closely related Laplacian linear system and apply the randomized Cholesky factorization. Two essential ingredients for achieving practical performance include a heuristic for sampling a clique and a fill-reducing reordering before factorization. The resulting sparse factorization is shown to outperform \ichol{} when both have roughly the same amount of fill-in. We view \rchol{} as a variant of standard incomplete Cholesky factorization. But unlike classical threshold-based dropping and level-based dropping, the sampling scheme in \rchol{} is an unbiased estimator: it randomly selects a subset of a clique and assigns them new weights. 
Interestingly, fill-reducing orderings are critical for the practical performance of \rchol{}, but is generally not effective for \ichol{}. In addition, the nested-dissection decomposition used in our parallel algorithm does not affect the performance of \rchol{}, but generally degrades the preconditioner quality of \ichol{}.

The described algorithm extends to the following two cases. The first  is that $A$ is an SPD matrix that has only non-positive off-diagonals (a.k.a., M-matrix). For such a matrix, there exists a positive diagonal matrix $D$ such that $DAD$ is SDDM~\cite{horn1994topics}, and then \rchol{} can be applied to $DAD$. The other  is that $A$ is the finite-element discretization of \cref{e:poisson} in a bounded open region with positive conductivity, i.e., $a(x)>0$. Such a matrix is generally SPD but not necessarily SDD, but there exists an analytical way to construct an SDD matrix whose preconditioner remains effective for $A$~\cite{boman2008solving}.

Three important directions for future research include:
\begin{itemize}
\item
Investigating variants of \cref{a:sample} to sample more edges in a clique, which leads to approximate Cholesky factorizations with more fill-in than the one computed by \rchol{}. Such approximations can potentially be more effective preconditioners for hard problems where the preconditioner based on \rchol{} converges slowly.





 \item
Parallelizing tasks for separators, especially for those at top levels. As \cref{f:parallel} shows, such tasks become the bottleneck of the parallel factorization time when a large number of threads are used. A naive method is to  apply the current parallel algorithm recursively on the (sparse) frontal matrices associated with those top separators.

\item
Extending the current framework combining Gaussian elimination with random sampling to unsymmetric matrices, which leads to an approximate LU factorization. See~\cite{cohen2018solving} for some progress in this direction.

\end{itemize}

\appendix

\section{Proof of \cref{th:complexity}} \label{s:a}

\begin{proof}
Consider the matrix/graph after an elimination step in \cref{a:rchol}; the number of non-zeros/edges decreases by 1. The reason is that at every step $n$ edges are eliminated and $n-1$ edges are added/sampled, where $n = |\mathcal{N}_k|$ is the number of neighbors or the number of non-zeros in the eliminated row/column excluding the diagonal. Since a random row/column is eliminated at every step, we have 
\[
\mathbb{E}[n] = \frac{M-k+1}{N-k+1}
\]
at the $k_\text{th}$ step. It is obvious to see that the computational cost and storage required by \cref{a:sample} is $\bigO(n)$ at every step. Therefore, the expected running time and the expected storage are both bounded by
\[
\sum_{k=1}^{N} \frac{M-k+1}{N-k+1}  < \sum_{k=1}^{N} \frac{M}{N-k+1}   < M \log N.
\]
\end{proof}

\section{Proof of equivalence in \cref{d:bsdd}} \label{s:b}

\subsection{Lemma}

\begin{lemma} \label{l:rank}
If matrix $A \in \mathbb{R}^{N \times N}$ is an irreducible SDD matrix, then $\mathrm{rank}(\hat{A}) \ge N-1$, where matrix  $\hat{A}$ defined in \cref{e:hat}.
\end{lemma}

\begin{proof}
Consider the following quadratic form given a nonzero $x \in \mathbb{R}^N$:
\[
x^\top \hat{A} \, x = \sum_{i,j} -a^n_{ij} (x_i - x_j)^2 + \sum_{i,j} a^p_{ij} (x_i + x_j)^2 \ge 0, 
\]
where $a^n_{ij}$ and $a^p_{ij}$ denote negative and positive off-diagonal entries in $A$, respectively. Suppose $x$ lies in the null space of $A$, we know that $x_i = x_j$ corresponding to every $a^n_{ij}$ and $x_i = - x_j$ corresponding to every $a^p_{ij}$. In addition, we know $x$ is entry-wise nonzero because $A$ is irreducible (underlying graph is connected). Therefore, we can find at most one such $x$ (up to a scalar multiplication), which implies that $\mathrm{rank}(\hat{A}) \ge N-1$.
\end{proof}

\subsection{Formal proof}

\begin{proof}
Assuming \textit{(a)} holds, we derive \textit{(c)}. There exists a nonzero $x \in \mathbb{R}^N$ such that $\hat{A} x = 0$. Consider the quadratic form
\[
x^\top \hat{A} \, x = \sum_{i,j} -a^n_{ij} (x_i - x_j)^2 + \sum_{i,j} a^p_{ij} (x_i + x_j)^2 = 0, 
\]
where $a^n_{ij}$ and $a^p_{ij}$ denote negative and positive off-diagonal entries in $A$, respectively. Hence, we know that $x_i = x_j$ corresponding to every $a^n_{ij}$ and $x_i = - x_j$ corresponding to every $a^p_{ij}$. In addition, we know $x$ is entry-wise nonzero because $A$ is irreducible (underlying graph is connected). Therefore, $x$ implies that  the graph $\G$ is 2-colorable in that all vertices $v_i$ corresponding to $x_i > 0$ have the same color while all vertices $v_i$ corresponding to $x_i < 0$ have the other color. 

Assuming \textit{(b)} holds, we derive \textit{(a)} and \textit{(c)} as follows. Without loss of generality (WLOG), suppose $D = \diag(\underbrace{1, \ldots, 1}_{n_1}, \underbrace{-1, \ldots, -1}_{n_2})$ and the matrix $A$ is partitioned as
\[
A = 
\begin{pmatrix}
A_{11} & A_{12}   \\
A_{21}  & A_{22}  \\
\end{pmatrix}
,
\]
where $A_{11} \in \mathbb{R}^{n_1 \times n_1}$ and $A_{22} \in \mathbb{R}^{n_2 \times n_2}$. Since $DAD$ has only non-positive off-diagonal entries,  $A_{11}$ and $A_{22}$ have non-positive off-diagonal entries, while $A_{12}$ and $A_{21}$ have non-negative entries. Hence, we know the following:
\begin{itemize}
\item
the vector $D \bone$ is in the null space of  $\hat{A}$, which is thus rank deficient. According to \cref{l:rank}, we know $\mathrm{rank}(\hat{A}) = N-1$.

\item
 the graph $\G$ is 2-colorable in that $v_1, v_2, \ldots, v_{n_1}$ have the first color, and $v_{n_1+1}, v_{n_1+2}, \ldots, v_{N}$ have the other color.
\end{itemize}

Assuming \textit{(c)} holds, we derive \textit{(b)}. WLOG, suppose $v_1, v_2, \ldots, v_{n_1}$ have the same color, which is different from the color that $v_{n_1+1}, v_{n_1+2}, \ldots, v_{N}$ have. In other words, matrix $A$ can be partitioned into
\[
A = 
\begin{pmatrix}
A_{11} & A_{12}   \\
A_{21}  & A_{22}  \\
\end{pmatrix}
,
\]
where $A_{11} \in \mathbb{R}^{n_1 \times n_1}$ and $A_{22} \in \mathbb{R}^{n_2 \times n_2}$ have non-positive off-diagonal entries, and $A_{12}$ and $A_{21}$ have non-negative entries. Therefore, the diagonal rescaling $D$ given by $\diag(\underbrace{1, \ldots, 1}_{n_1}, \underbrace{-1, \ldots, -1}_{n_2})$ satisfies that $DAD$ has only non-positive off-diagonal entries.
\end{proof}

\section{Proof of \cref{th:irreducible}} \label{s:c}

\begin{proof}
WLOG, assume $A = \hat{A} \in \mathbb{R}^{N \times N}$; in other words, every diagonal entry equals to the sum of the absolute value of off-diagonal entries on the same row/column. Suppose there exists a nonzero vector in the null space of $\tilde{A} \in \mathbb{R}^{2N \times 2N}$, i.e.,
\[
\begin{pmatrix}
A_d+A_n & -A_p \\
-A_p & A_d+A_n
\end{pmatrix}
\begin{pmatrix}
x_1 \\
x_2
\end{pmatrix}
=0,
\]
where $x_1,x_2 \in \mathbb{R}^N$. It is easy to see that
\begin{align*}
(A_d+A_n + A_p)(x_1 - x_2) &= 0, \\
(A_d+A_n - A_p)(x_1 + x_2) &= 0. 
\end{align*}
Since $A = A_d+A_n + A_p$ is an irreducible non-bipartite SDD matrix, we know $\mathrm{rank}(\hat{A}) = \mathrm{rank}(A) = N$. Hence, $x_1 = x_2$. It is straightforward to verify that $A_d+A_n - A_p$ is a Laplacian matrix, and thus $x_1=x_2 \in \text{span}\{\bone\}$. Therefore, we know $\mathrm{rank}(\tilde{A}) = 2N-1$, which implies that Laplacian matrix $\tilde{A}$ is irreducible.

\end{proof}

\section{High-contrast coefficients for VC-Poisson} \label{s:e}
One instance of the random coefficients constructed in \cref{e:a} is shown in \cref{f:a}.
\begin{figure}
\begin{center}
\scalebox{0.2}{\includegraphics{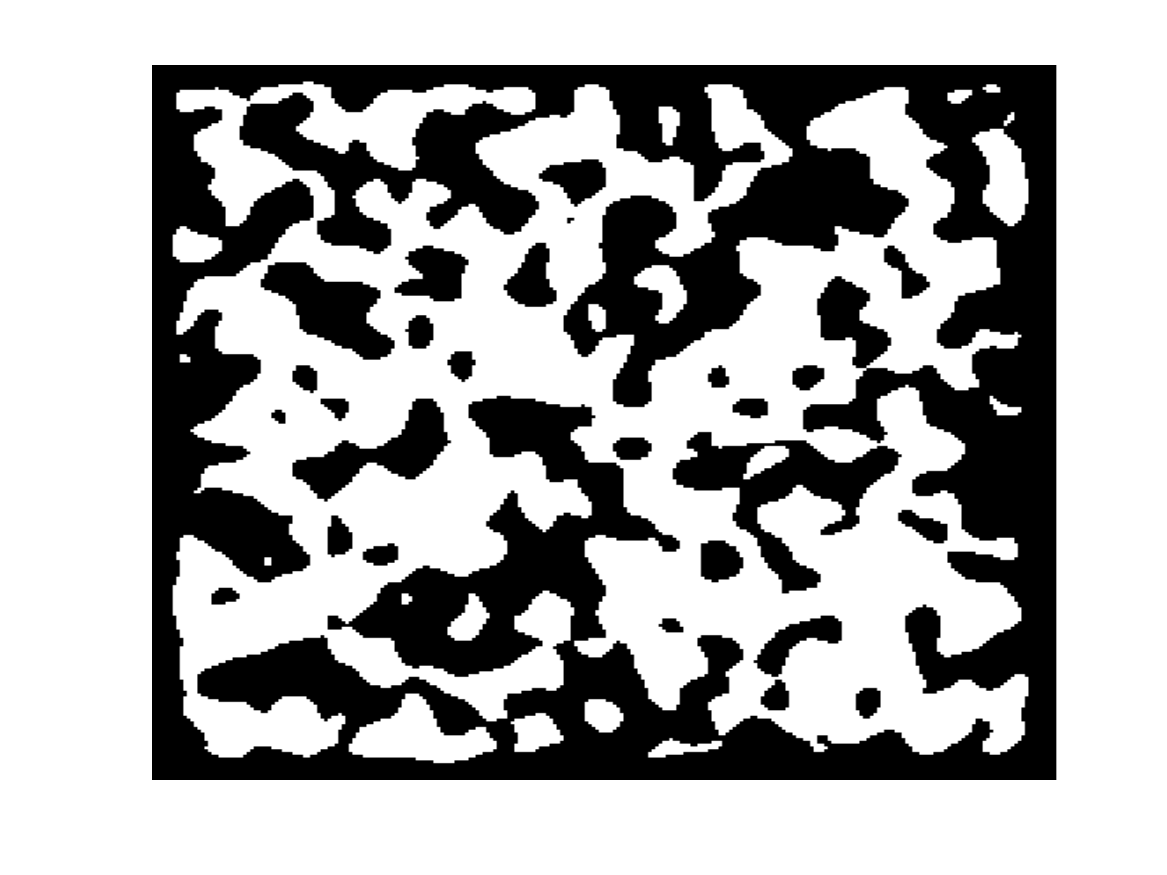}} 
\caption{Example of the high-contrast coefficients for the variable-coefficient Poisson's equation in the unit square with a 2D grid $256 \times 256$.}
\label{f:a}
\end{center}
\end{figure}

\section{{Results of parallel sparse triangular solve}} \label{s:d}

\cref{t:psolve} shows parallel timing results of the parallel sparse triangular solve. The cholesky factor $G$ was stored in the compressed sparse column format. Therefore, the upper triangular solve involving $G^\top$ was implemented in a straightforward way by a preorder traversal of the tree data structure used in \rchol{}; see \cref{s:parallel}. The lower triangular solve was implemented using a postorder traversal of our tree data structure. We implemented the parallel lower solve using  an asynchronous approach, where the two child nodes updates the data owned by their parent asynchronously following ideas in~\cite{glusa2020scalable,chow2015fine}.

\begin{table}
    \centering 
    \caption{\em Parallel sparse triangular solve (per iteration) on an Intel Cascade Lake that has 112 cores on four sockets. The matrices are from discretizing Poisson's equation on a 3D regular grid with the standard 7-point stencil.}
    \label{t:psolve}
    \begin{tabular}{c|ccc|ccc|ccc} 
    \toprule
    \multirow{2}{*}{$p$}
    & \multicolumn{3}{c|}{$N=128^3$} & \multicolumn{3}{c|}{$N=256^3$}  & \multicolumn{3}{c}{$N={512^3}$} \\
    & $t_\text{lower}$ & $t_\text{upper}$ & $n_{it}$ & $t_\text{lower}$ & $t_\text{upper}$ & $n_{it}$ & $t_\text{lower}$ & $t_\text{upper}$ & $n_{it}$  \\ 
    \midrule
    1 & 0.0400 & 0.0430 & 50 & 0.409 & 0.409 & 57 & 5.59 & 4.49 & 64 \\
    2 & 0.0499 & 0.0536 & 50 & 0.333 & 0.348 & 57 & 2.93 & 2.69 & 67 \\
    4 & 0.0423 & 0.0446 & 50 & 0.199 & 0.197 & 58 & 1.51 & 1.31 & 65 \\
    8 & 0.0280 & 0.0301 & 53 &  0.157 & 0.161 & 54 & 0.962 & 0.814 & 64  \\
    16 & 0.0177 & 0.0200 & 49 & 0.136 & 0.136 & 59 & 0.730 & 0.536 & 65 \\
    32 & 0.0123 & 0.0140 & 49 & 0.113 & 0.121 & 55 & 0.603 & 0.404 & 64 \\
    64 & 0.0126 & 0.0107 & 50 & 0.104 & 0.104 & 57 & 0.653 & 0.429 & 67 \\
    \bottomrule
    \end{tabular}
\end{table}





\bibliographystyle{siamplain}
\bibliography{biblio}

\end{document}